\documentclass[a4paper, 11pt, reqno]{amsart}

\usepackage{amsmath,amssymb,amscd,amsthm,amsxtra, esint,bbm}

\usepackage{mathtools}
\usepackage[linktoc=page,colorlinks=true,linkcolor=blue,citecolor=red]{hyperref}

\usepackage{enumitem}

\usepackage{marginnote}

\usepackage{tikz}

\usepackage{cases}%%%

\makeatletter
\renewcommand{\subjclassname}{%
  \textup{1991} Mathematics Subject Classification}
\@xp\let\csname subjclassname@1991\endcsname \subjclassname
\@namedef{subjclassname@2000}{%
  \textup{2000} Mathematics Subject Classification}
\@namedef{subjclassname@2010}{%
  \textup{2010} Mathematics Subject Classification}
  \@namedef{subjclassname@2020}{%
  \textup{2020} Mathematics Subject Classification}
\makeatother

\allowdisplaybreaks[2]

\newtheorem{theorem}{Theorem} [section]

\newtheorem{lemma}[theorem]{Lemma}
\newtheorem{proposition}[theorem]{Proposition}
\newtheorem{remark}[theorem]{Remark}

\newtheorem{definition}[theorem]{Definition}

\newtheorem{innerassumption}{Assumption}
\newenvironment{assumption}[1]
  {\renewcommand\theinnerassumption{#1}\innerassumption}
  {\endinnerassumption}

\newcommand{\noi}{\noindent}
\newcommand{\Z}{\mathbb{Z}}
\newcommand{\R}{\mathbb{R}}
\newcommand{\C}{\mathbb{C}}

\newcommand{\T}{\mathbb{T}}
\newcommand{\TT}{\mathcal{T}}

\let\Re=\undefined\DeclareMathOperator*{\Re}{Re}

\let\P= \undefined
\newcommand{\P}{\mathbf{P}}

\renewcommand{\L}{\mathcal{L}}

\newcommand{\NN}{\mathcal{N}}
\newcommand{\Nb}{\mathbf{N}}
\newcommand{\Fb}{\mathbf{F}}
\newcommand{\Eb}{\mathbf{E}}
\newcommand{\X}{\mathbf{X}}
\newcommand{\Y}{\mathbf{Y}}

\newcommand{\RR}{\mathcal{R}}

\newcommand{\K}{\mathbb{K}}

\newcommand{\F}{\mathcal{F}}

\newcommand{\ind}{\mathbf{1}}

\newcommand{\eps}{\varepsilon}

\newcommand{\ft}{\widehat}
\newcommand{\wt}{\widetilde}
\newcommand{\cj}{\overline}
\newcommand{\dx}{\partial_x}

\newcommand{\dt}{\partial_t}

\newcommand{\les}{\lesssim}

\newcommand{\jb}[1]
{\langle #1 \rangle}

\renewcommand{\ind}{\mathbf 1}

\renewcommand{\S}{\mathcal{S}}

\newcommand{\ZZ}{\mathcal{Z}}

\newcommand{\M}{\mathcal{M}}

\newcommand{\N}{\mathbb{N}}
\renewcommand{\SS}{\mathbb{S}}

\newtheorem*{ackno}{Acknowledgements}

\numberwithin{equation}{section}
\numberwithin{theorem}{section}

\begin{document}
\baselineskip = 15pt

\title[Strongly non-resonant modulated dispersive PDEs]
{Regularization by noise for some strongly non-resonant modulated dispersive PDEs}

\author{Tristan Robert}

\address{\small{Université de Lorraine, CNRS, IECL, F-54000 Nancy, France}}

\email{tristan.robert@univ-lorraine.fr}

\subjclass[2020]{35Q53,35Q55,60L50}

\keywords{Regularization by noise, nonlinear dispersive PDEs}

\begin{abstract}
In this work, we pursue our investigations on the Cauchy problem for a class of dispersive PDEs where a rough time coefficient is present in front of the dispersion. We show that if the PDE satisfies a strong non-resonance condition (Theorem~\ref{THM:LWP-NR}), eventually up to a completely resonant term (Theorem~\ref{THM:LWP-R}), then the modulated PDE is well-posed at any regularity index provided that the noise term in front of the dispersion is irregular enough. This extends earlier pioneering work of Chouk-Gubinelli and Chouk-Gubinelli-Li-Li-Oh to a more general context. We quantify the irregularity of the noise required to reach a given regularity index in terms of the regularity of its occupation measure in the sense of Catellier-Gubinelli. As examples, we discuss the cases of dispersive perturbations of the Burger's equation, including the dispersion-generalized Korteweg-de Vries and Benjamin-Ono equations, the intermediate long wave equation, the Wick-ordered modified dispersion-generalized Korteweg-de Vries equation, and the fifth-order Korteweg-de Vries equation. We also treat the completely non-resonant nonlinear Schr\"odinger equation and the Wick-ordered fractional cubic nonlinear Schr\"odinger equation, all with periodic boundary conditions.
\end{abstract}

%\date{\today}
%%
%
\maketitle

\vspace{-5mm}

\tableofcontents

\section{Introduction}
In this paper, we continue the investigations started in \cite{DBD,DT,DR,CG1,CG2} and in our previous work \cite{ModulatedNLS} on the well-posedness of abstract dispersive PDEs with modulated dispersion:
\begin{align}\label{EQ}
\begin{cases}{\displaystyle \dt u+\frac{d W}{dt}\L u+ \NN(u)=0,}\\
u_{|t=0}=u_0
\end{cases}
~~u:(t,x)\in\R\times \M\to \K,
\end{align}
where $\M$ is some boundaryless manifold to be specified, $\K\in\{\R;\C\}$, $\L$ is a skew-adjoint operator, $\NN$ is a nonlinear function of $u$ possibly depending also on its derivatives, and $W_t:\R\to\R$ is a given continuous but not necessarily differentiable function. We are thus interested in the effect of the rough modulation $W_t$ on the behaviour of the PDE \eqref{EQ} compared to that of the ``deterministic'' PDE
\begin{align}\label{EQ:deterministic}
\begin{cases}
\dt u +\L u +\NN(u)=0,\\
u(t=0)=u_0.
\end{cases}
\end{align}

The model \eqref{EQ} has been particularly studied when the original equation \eqref{EQ:deterministic} corresponds to the nonlinear Schr\"odinger (NLS) equation, namely $\K=\C$, $\L=i\Delta$, and $\NN(u)=|u|^{p-2}u$ for some $p>2$. We refer to the introduction in \cite{ModulatedNLS} for more physical background and references regarding the analysis of the modulated NLS equation \eqref{EQ}.

\subsection{Regularization by noise}
In this follow-up work after \cite{ModulatedNLS}, we are particularly interested in the so-called \emph{regularization by noise} phenomenon that one can observe for \eqref{EQ} compared to \eqref{EQ:deterministic}. Indeed, building on the breakthrough work of Chouk and Gubinelli in \cite{CG2}, our main results in Theorem~\ref{THM:LWP-NR} and~\ref{THM:LWP-R} below can be informally stated as follows.
\begin{theorem}[Informal version]\label{THM:informal}
If the PDE \eqref{EQ:deterministic} satisfies a strong non-resonant condition, then the modulated equation \eqref{EQ} is well-posed in $H^s(\M)$ for \textbf{\emph{any}} $s\in\R$ provided that $W_t$ is irregular enough depending on $s$.
\end{theorem}

See below for the proper definition of strong non-resonance. Note that for nonlinear PDEs, one can in general compute a critical exponent $s_c$ dictated by scaling considerations, for which one expects ill-posedness for initial data of regularity $s<s_c$. Thus, in this respect we can see that adding the rough modulation term in front of the dispersion has a strong regularizing\footnote{Actually, the term $\frac{dW_t}{dt}$ in \eqref{EQ} changes the scaling of the equation compared to \eqref{EQ:deterministic}, which explains why the well-posedness can be obtained in a wider range when $W_t$ is rough, compared to the case $W_t=t$. However, this is in contrast with the results of \cite{ModulatedNLS}, where we showed that at the level of linear space-time estimates, one cannot hope to benefit from the change of scaling due to the noise to get improvement on Strichartz type estimates.}  effect on the dynamics of \eqref{EQ:deterministic}.

The regularizing effect of adding a noise in a singular dynamics has been observed for quite a while now, and for various models. In the case of stochastic perturbations of differential equations, there is by now a very large body of research regarding the regularization effect of the noise on the dynamics of ODEs generated by rough vector fields. In this context, if the vector field is too rough, one loses uniqueness of the solution without the noise, but the latter is restored in presence of the noisy perturbation. Such a regularization effect on dynamics of ODEs was observed in pioneering works of Zvonkin \cite{Zvonkin}, Veretennikov \cite{Veretennikov}, and Krylov-R\"ockner \cite{KR}, to name but a few. From then, a more systematic study of such a regularization effect for SDEs lead to an ever growing literature. One direction related to the present work is in particular investigated by Davie \cite{Davie} and Cattelier-Gubinelli \cite{CaG}: while earlier works relied more or less heavily on the stochastic properties of the noisy perturbation, the work of Davie showed that the regularization effect can actually be observed path-wise, i.e. for a.e. realization of the noise. Catellier and Gubinelli went one step further in this direction and proved that this phenomenon is actually entirely deterministic, and relies on the irregularity properties of the noise, observed through the regularity properties of its occupation measure; see Definition~\ref{DEF:occup} below. This approach of \emph{noiseless} regularization by noise was further studied by Galeati-Gubinelli \cite{GaGu1,GaGu2}. See also \cite{GaGe} for a sharp result (up to endpoint) on the regularity of the drift allowed in the case of a fractional Brownian perturbation, and references therein for a more comprehensive introduction to the SDE case.

Moving to the PDE setting, using the relation between ODEs and transport equations through the method of characteristic, the important work of Flandoli, Gubinelli, and Priola \cite{FGP} showed that such a regularization effect can also be observed for PDEs subject to stochastic perturbations; in this case, linear transport equations with a rough drift. We refer to the monograph \cite{Flandoli} for a more detailed account on these results. In recent years, more results appeared on regularization by noise for various kinds of PDEs: stochastic scalar conservation laws \cite{ChoukGess,GessMaurelli}, stochastic Hamilton-Jacobi equations \cite{GGLS}, stochastic parabolic equations \cite{GP,BM}, stochastic wave equation \cite{BHR}, among other. See the review \cite{GessReview}, the introduction of \cite{GaleatiThesis}, and references therein for a more detailed account on the subject.

Regarding nonlinear \emph{dispersive} PDEs, few results studied the regularizing effect of the noise in the equation. It was shown for various models that non conservative noise can improve the long-time behavior of the PDE such as global well-posedness and scattering \cite{BRZ,HRZ,HRSZ1,HRSZ2}. The case of noisy dispersion as in \eqref{EQ} has been studied in \cite{DBD,DT,DR,CG1,CG2,ModulatedNLS}. In particular, in \cite{DT}, Debussche and Tsutsumi showed that the noisy dispersion allows to prevent blow-up for the focusing mass-critical 1d NLS with large data. This was extended to a path-wise result in \cite{DR}, and a completely deterministic one in \cite{CG1} (see also \cite{GaleatiThesis}). In \cite{ModulatedNLS}, we found the minimal assumption on the regularity of the occupation measure of the noise to ensure that this phenomenon occurs. But we also proved that one \emph{cannot} improve the range of exponents in the local in time Strichartz estimate for the modulated dispersion, which shows that a regularizing effect of the dynamics, in terms of recovering well-posedness in larger spaces, cannot rely only on \emph{linear} estimates. Similarly, Stewart \cite{Stewart} showed that for the 1d periodic quintic NLS, the same obstruction to well-posedness in $L^2(\T)$ with a locally Lipschitz dependence on initial data remains for the modulated equation, showing that the regularizing effect of the noisy dispersion is not systematic.

However, Chouk and Gubinelli \cite{CG2} made a breakthrough by showing that for the periodic modulated Korteweg - de Vries (KdV) equation, corresponding to \eqref{EQ} with $\L=\dx^3$ and $\NN(u)=\dx(u^2)$, one has indeed a very strong regularizing effect of the noise. Namely, they showed that the informal statement of Theorem~\ref{THM:informal} is true for the modulated KdV equation. Their proof highlights how the oscillations of a very irregular noise can collaborate with the \emph{multilinear} oscillations of the nonlinear dispersive PDE to produce a strong regularizing effect. In the present work, we aim at showing that the strong regularizing effect showed in \cite{CG2} holds for a large class of models being so-called \emph{strongly non-resonant} (following the terminology of \cite{MV}), and even for completely resonant perturbations of such models. We hope that this will clarify the picture between the general case, where there is no regularization effect at the level of linear estimates, and the strongly non-resonant one; see also \cite{JEDP} for a further discussion in this direction and announcements of the results of the current paper. 

The counter-examples of \cite{ModulatedNLS} regarding the lack of improvement on linear estimates, and of \cite{Stewart} regarding the lack of regularization of the flow map, indicate that the regularization effect in the \emph{non} strongly non-resonant case, if any, is more subtle than the one highlighted in \cite{CG2} and that we investigate here. See also \cite[Theorem 1.13]{ModulatedNLS} and \cite{ModulatedKP} where the noise improves \emph{transverse} interactions.

\begin{remark}
\rm  
Upon completion of this work, we learned of the very recent update \cite{CGLLO} by Chouk, Gubinelli, Li, Li, and Oh of the breakthrough work \cite{CG2}. In this new version of their results, the authors exemplify Theorem~\ref{THM:informal} by showing that the strong regularization by noise observed for the periodic modulated KdV equation in \cite{CG2} also holds for the periodic modulated Benjamin-Ono (BO) equation and the periodic intermediate long wave (ILW) equation. We find similar conclusions for these models entering the framework of Theorem~\ref{THM:LWP-NR} below, but our proofs are slightly different. We refer to the examples of Subsubsection~\ref{EX:NR} and Subsection~\ref{SUBS:proof} below for a further discussion.
\end{remark}

\subsection{Assumptions}

Before stating our results, we need to introduce the setting and set of assumptions on the noise $W_t$ and the abstract model \eqref{EQ:deterministic}. The first one concerns the irregularity assumptions on $W_t$, that we will quantify through its occupation measure.

\subsubsection{On the irregularity of $W_t$}
\begin{definition}\label{DEF:occup}
\rm ~~
\begin{itemize}
\item Let $W:\R\to \R$ be a continuous path. For all time interval $I\subset \R$, its \textbf{occupation measure} is defined as the measure on Borel sets $A\in\mathcal{B}(\R)$
\begin{align}\label{mu}
\mu_I(A):=\mathrm{Leb}\big(t\in I,~W_t\in A\big)=\int_I\mathbf{1}_A(W_t)dt.
\end{align}
\item In the following we assume that $\mu_I\ll dz$ for all interval $I\subset \R$. Then ${\displaystyle\frac{d\mu_I}{dz}}$ is the \textbf{local time} of $W_t$ in $I$.
\end{itemize}
\end{definition}

In this paper, we will use the following assumptions on the regularity of $\frac{d\mu_{[0;\cdot]}}{dz}$ and the dispersive properties of \eqref{EQ:deterministic}. 

\begin{assumption}{\textup{\textbf{(A0)$_{(\rho,\gamma)}$}}}\label{A0}
\rm ($(\rho,\gamma)$-irregularity of $W_t$) $W_t$ is a continuous function which is $(\rho,\gamma)$-irregular in the sense of \cite[Definition 1.3]{CaG}, i.e. such that $$\frac{d\mu_{[0;\cdot]}}{dz}(\cdot)\in C^\gamma_t\big(\R;\F L^{\rho,\infty}(\R)\big).$$
 \end{assumption}
 
 The relation between (ir)regularity of a path $W_t$ and the regularity of its occupation has been thoroughly investigated since the early 70's for deterministic paths, see the review paper \cite{Review} and references therein, and much earlier for stochastic $W_t$, especially in the case of the Brownian motion \cite{Levy}, for which we refer to the monograph \cite{RY}. While earlier works measured the regularity of the local time in H\"older spaces \cite{Levy,Review,RY}, more recent studies investigated other choices of topology: Fourier-Lebesgue in Catellier-Gubinelli \cite{CaG}, or Besov in Romito-Tolomeo \cite{RT}. See also the discussion in \cite[Subsection 1.2]{ModulatedNLS}.
 
As an example, the ``deterministic'' path $W_t=t$ satisfies \ref{A0} if and only if $\gamma\in [0;1]$ and $\rho\le 1-\gamma$. The same holds for more regular bijections, see \cite[Proposition 1.4]{CG1}. Moreover, for any $\eta\in(0;1)$, the assumption \ref{A0} for some $\rho>1$ and $\gamma>1-\eta$ implies that $W_t$ is not $\eta$-H\"older continuous at any $t$, see \cite{Review}.

\begin{remark}
\rm

When $W_t$ is a fractional Brownian motion of Hurst parameter $H\in(0;1)$, Catellier and Gubinelli \cite[Theorem 1.4]{CaG} proved that \ref{A0} holds a.s. for $\gamma\in[\frac12;1]$ and $\rho<\frac{1-\gamma}{H}$. Therefore, in Theorems~\ref{THM:LWP-NR} and~\ref{THM:LWP-R} below, the assumptions on $\rho$ (being sufficiently large) can be expressed in this case as an assumption on $H$ (being sufficiently small). Since in this case $W_t$ is known to be a.s. in $C^{H-}\setminus C^H$, the assumptions \ref{Arho} and \ref{Arho*} below on $\rho$ being large enough are indeed assumptions on the fractional Brownian motion $W_t$ being irregular enough in this case, as advertised in the informal statement of Theorem~\ref{THM:informal}.
\end{remark} 
 
 \subsubsection{Assumptions on the dispersion and the nonlinearity}
 We now state the assumptions on the dispersive model \eqref{EQ}, and in particular its \emph{strongly non-resonant} character.
 
\begin{assumption}{\textbf{\textup{(A1)}}}\label{A1}
\rm (Domain and dispersion) $\M=\T^d$ for some $d\in\N$. Then $\L$ is a skew-adjoint operator given by a Fourier multiplier $\varphi\in C(\R^d;\R)$ such that for any $u\in \S(\T^d)$ and $k\in\Z^d$, $\widehat{\L u}(k)=i\varphi(k)\widehat{u}(k)$.
\end{assumption}
%\begin{assumption}{\textbf{\textup{(A2)}}}\label{A2}
%\rm (Conservation laws) The mass $u\mapsto \frac12\|u\|_{L^2_x}^2$ is a conservation law of \eqref{EQ:deterministic}.
%\end{assumption}
\begin{assumption}{\textbf{\textup{(A2)}}}\label{A3}
\rm (Algebraic nonlinearity) there is an integer $m\ge 2$, a choice of signs $\{\pm_j\}_{j=0}^m\in\{\pm\}^{m+1}$, and a multiplier $\widehat{\NN}:\Gamma^{m+1}(\{\pm_j\})\to \C$, where
\begin{align*}
\Gamma^{m+1}(\{\pm_j\})=\Big\{(k_0,\dots,k_m)\in(\Z^d)^{m+1},~~\sum_{j=0}^m\pm_jk_j=0\Big\},
\end{align*}
such that for any $k_0\in\Z^d$ and $u\in\S(\T^d)$,
\begin{align*}
\widehat{\NN(u)}(k_0)=\sum_{\substack{(k_1,\dots,k_m)\in(\Z^d)^m\\\sum_{j=0}^m\pm_jk_j=0}}\widehat{\NN}(k_0,\dots,k_m)\prod_{j=1}^m\widehat{u}_{k_j}^{\pm_j},
\end{align*}
with the convention $\widehat{u}_{k_j}^+=\widehat{u}_{k_j}$ and $\widehat{u}_{k_j}^-=\cj{\widehat{u}_{k_j}}$, and $\widehat{u}_k$ being the $k$-th Fourier\footnote{In the remaining of the paper, we neglect various factors of $2\pi$ that appear in the normalization of the Fourier transform, as they play no role in the analysis.} coefficient of $u$:
\begin{align*}
\widehat{u}_k = \frac1{(2\pi)^d}\int_{\T^d}u(x)e^{-ik\cdot x}dx.
\end{align*}
\end{assumption}
\begin{assumption}{\textbf{\textup{(A3)}}}\label{A4}
\rm (Symbol of the non-linearity) there are non-negative $\beta=(\beta_0,\dots,\beta_m)\in \R_+^{m+1}$ and $C>0$ such that for any $(k_0,\dots,k_m)\in\Gamma^{m+1}(\{\pm_j\})$,
\begin{align*}
|\widehat{\NN}(k_0,\dots,k_m)|\le C\prod_{j=0}^m\jb{k_j}^{\beta_j}.
\end{align*}
\end{assumption}
\begin{assumption}{\textbf{\textup{(A4)}}}\label{A5}
\rm (Strongly non-resonant equation) there are $\alpha=(\alpha_1,\alpha_2)\in (0;\infty)\times[0;\infty)$ and $c>0$ such that for any $(k_0,\dots,k_m)\in\Gamma^{m+1}(\{\pm_j\})$ in the support of $\widehat{\NN}$, the resonance function
\begin{align}\label{resonance}
\Phi_{\L,\NN}(k_0,\dots,k_m)=\sum_{j=0}^m\pm_j\varphi(k_j)
\end{align}
satisfies
\begin{align}\label{NR}
\big|\Phi_{\L,\NN}(k_0,\dots,k_m)\big|\ge c \jb{k_0^*}^{\alpha_1}\jb{k_2^*}^{\alpha_2},
\end{align}
where $|k_0^*|\ge |k_1^*|\ge\dots\ge |k_m|^*$ denotes the decreasing rearrangement of $(k_0,\dots,k_m)\in\Gamma^{m+1}(\{\pm_j\})$. In particular $|k_0^*|\sim|k_1^*|$.
\end{assumption}
\begin{assumption}{\textbf{\textup{(A4*)}}}\label{A5*}
\rm (Completely resonant/Strongly non-resonant equation) 
%the equation \eqref{EQ:deterministic} is Hamiltonian with Hamiltonian $H$ and Poisson bracket $\{H,G\}=\langle J\nabla H,\nabla G\rangle_{L^2_x}$ for some skew-symmetric operator $J$. Namely, \eqref{EQ:deterministic} can be written as
%\begin{align}\label{Hamiltonian}
%\dt \ft u(t,k_0)=i\ft J(k_0)\partial_{\cj{\ft u(k_0)}}H(u),~~\forall k_0\in\Z^d,
%\end{align}
%with $\ft J(k_0)\in\R$ for all $k_0\in\Z^d$. Moreover, we assume that $H=H_0+H_1$ with:\\
%\textup{(i)} $\L=J\nabla H_0$;\\ 
%\textup{(ii)} $\NN=J\nabla H_1=\NN_0+J\nabla\RR$, with $\NN_0$ satisfying \textbf{(A3)-(A4)$_\beta$-(A5)$_\alpha$};\\
%\textup{(iii)} $\nabla\RR$ only satisfies {\ref{A3}}, is bounded on $H^{s_\RR}(\T^d)$ for some $s_\RR\in\R$ together with the estimate 
%\begin{align}\label{tame}
%\|\nabla\RR(u_1,\dots,u_m)\|_{H^\sigma}\le C\sum_{j=1}^m\|u_j\|_{H^\sigma}\prod_{j'\neq j}\|u_{j'}\|_{H^{s_\RR}}
%\end{align}
%for any $\sigma\ge s_\RR$;\\
%\textup{(iv)} $\RR$ only depends on the actions: $\RR(u)=\RR\big(\{|\ft u(k_0)|^2\}_{k_0}\big)$.
The nonlinearity $\NN$ in \eqref{EQ:deterministic} can be decomposed as $\NN=\NN_0+\RR$, with:\\
\textup{(i)} $\NN_0$ satisfies \ref{A3}-\ref{A4}$_{\beta}$-\ref{A5}$_{\alpha}$;\\
\textup{(ii)} $\RR$ only satisfies \ref{A3}-\ref{A4}$_{\beta_\RR}$, but has the completely resonant structure:
\begin{align}\label{Resonant}
\begin{cases}
\text{$m$ is odd,}\\
\text{ $\pm_j=(-1)^j$,}\\
\text{$\ft\RR(k_0,\dots,k_m)=\mathbf{1}_{k_0=k_1=\dots=k_m}\ft\RR(k_0,\dots,k_0)$}\\
\text{$\ft\RR(k_0,\dots,k_0)\in i\R$.}
\end{cases}
\end{align}
\end{assumption}
In the assumption above and in the rest of the paper, we will write \ref{A4}$_{\beta}$, \ref{A5}$_{\alpha}$, and \ref{A5*}$_{\alpha,\beta,\beta_\RR}$ to specify the parameters given by each assumption. See also {\ref{Arho}$_{\alpha,\beta,s}$} below.

\begin{remark}\label{REM:Td}
\rm
We only consider the case $\M=\T^d$ in \ref{A1}, since on $\M=\R^d$ the strong non-resonance property \eqref{NR} usually fails due to (very) low frequencies. We mainly restrict to $d=1$ in Subsubsection~\ref{EX:NR} below in order to give examples of physically relevant models, as \ref{A5*}$_{\alpha,\beta,\beta_\RR}$ typically fails for higher dimensions due to high multiplicities of eigenvalues for $\L$. This is even more true on more general closed manifolds $\M$ (e.g. $\M=\SS^d$), since there is also a lack of convolution structure.
\end{remark}

%\begin{remark}
%\rm
%In assumption \ref{A5*}, we only assume that the integrable part $\ft \RR(u)_{k_0}$ contains only decoupled terms in the actions, i.e. no terms like $\sum_{k_1}|\ft u_{k_0}|^{2m_0}\ft u_{k_0}|\ft u_{k_1}|^{2m_1}$ for some $m_1\in\N^*$. This is because in order to work at
%\end{remark}

%
%\begin{remark}
%\rm
%Even if the corresponding deterministic equation \eqref{EQ:deterministic} is Hamiltonian with a conserved energy, the latter is not conserved anymore for \eqref{EQ}.
%\end{remark}

\subsection{Main results}

Our main results are then the following theorems, showing that under a strongly non-resonant assumption, \textit{enough irregularity} of the noise $W_t$ in \eqref{EQ} allows to improve on the well-posedness theory compared to the deterministic equation \eqref{EQ:deterministic}. 

In order to state our results, we introduce the following assumption on $\rho$, depending on some parameters $\alpha\in(0;\infty)\times[0;\infty)$, $\beta\in \R^{m+1}$, and $s\in\R$, to quantify how irregular $W_t$ must be:\\ 
\begin{assumption}{\textup{\textbf{(A$_\rho$)}}}\label{Arho} (Enough irregularity of $W_t$)
\begin{align}\label{rho}
\begin{cases}
\rho\ge \alpha_1^{-1}\big(\beta_0+\beta_1^*),\\
\rho\ge \alpha_1^{-1}\big(\beta_1^*+\beta_2^*-2s),\\
{\displaystyle \rho>(\alpha_1+\alpha_2)^{-1}\big(\beta_1^*+\beta_2^*-2s+\sum_{\ell=2}^{L}(\beta_{\ell+1}^*+\frac{d}2-s)\big),}\\
\qquad\qquad\qquad\qquad \text{ for any } 2\le \ell_0\le m\text{ and }2\le L\le \ell_0-1,\\
{\displaystyle \rho>(\alpha_1+\alpha_2)^{-1}\big(\beta_1^*+\beta_2^*+\beta_0+\frac{d}2-s+\sum_{\ell=2}^{L}(\beta_{\ell+1}^*+\frac{d}2-s)\big)},\\
\qquad\qquad\qquad\qquad\text{ for any }2\le \ell_0\le m\text{ and }\ell_0+1\le L\le m,\\
\end{cases}
\end{align}
where $\beta_1^*\ge \beta_2^*\ge\dots\ge \beta_m^*$ is the decreasing rearrangement of $(\beta_1,\dots,\beta_m)\in\R_+^m$.
\end{assumption}

\subsubsection{\textbf{The strongly non-resonant case}}

Our first result deals with the purely strongly non-resonant case.
\begin{theorem}\label{THM:LWP-NR}
Assume that {\ref{A0}-\ref{A1}-\ref{A3}-\ref{A4}$_{\beta}$-\ref{A5}$_{\alpha}$} hold for some $\gamma\in(\frac12;1)$ and $\rho>0$. Then \eqref{EQ} is locally well-posed in $H^s(\T^d)$ for any $s\in\R$ provided that $\rho$ satisfies \textup{{\ref{Arho}$_{\alpha,\beta,s}$}}. More precisely:\\
\textup{(i)} For any $u_0\in H^s(\T^d)$, there exists $T=T(\|u_0\|_{H^s})>0$ such that \eqref{EQ} has a unique mild solution $u\in  C([0;T);H^s(\T^d))\cap \X^s_T$ satisfying $u_{|t=0}=u_0$;\\
\textup{(ii)} There is persistence of regularity: if $u_0\in H^\sigma(\T^d)$, $\sigma\ge s$, then $u\in C([0;T);H^\sigma(\T^d))\cap \X^\sigma_T$;\\
\textup{(iii)} The flow map $u_0\in H^s(\T^d)\mapsto u\in C([0;T);H^s(\T^d))\cap \X^s_T$ is Lipschitz continuous.
%\textup{(iv)} If \ref{A2} also holds and $\rho$ satisfies \ref{Arho}$_{\alpha,\beta,0}$, then for any $s\ge 0$, the local solution $u$ can be extended globally in time.
\end{theorem}
The definition of the functional spaces in the above statement are given in Section~\ref{SEC:Prelim}, see also the discussion in Subsection~\ref{SUBS:proof} below.

As advertised above, Theorem~\ref{THM:LWP-NR} gives a quantified version of the informal statement of Theorem~\ref{THM:informal}, and can be seen as the generalization to a broader context of the regularizing effect discovered by Chouk and Gubinelli when working on the modulated KdV and mKdV models \cite{CG2}, and also for BO and ILW in \cite{CGLLO}. Actually, Theorem~\ref{THM:LWP-NR} shows a \emph{double} regularizing effect of the noise in \eqref{EQ} compared to \eqref{EQ:deterministic} for strongly non-resonant models: while Theorem~\ref{THM:LWP-NR}~(i) indeed discusses the improvement for \eqref{EQ} over \eqref{EQ:deterministic} in the range of regularity exponents $s$ for the well-posedness theory to hold, Theorem~\ref{THM:LWP-NR}~(iii) also quantifies a regularizing effect at the level of the regularity of the flow map. Indeed, when the dispersive effect is not strong enough compared to the nonlinearity in \eqref{EQ:deterministic}, it may happen that the flow map stops being regular below a certain threshold $s_{\mathrm{reg}}\ge s_c$ of regularity above the scaling critical one, and can even be only continuous at any regularity. While one would think that the change of scaling due to the noise in \eqref{EQ} simply shifts the values of $s_c$ and accordingly of $s_{\mathrm{reg}}$, Theorem~\ref{THM:LWP-NR} applies to some models for which $s_{\mathrm{reg}}=\infty$ (the flow map is never regular); see the examples of Subsubsection~\ref{EX:NR} below. For such models, we can see that Theorem~\ref{THM:LWP-NR}~(iii) provides another strong regularizing effect since for \eqref{EQ}, $s_{\mathrm{reg}}\le s$ is allowed to be very low in Theorem~\ref{THM:LWP-NR}.

\subsubsection{\textbf{Examples for Theorem~\ref{THM:LWP-NR}}}\label{EX:NR}
We now discuss some examples of models which satisfy the assumptions of Theorem~\ref{THM:LWP-NR} and that we advertised in \cite{JEDP}. Since the main assumption \eqref{NR} typically fails as soon as $d$ or $m$ are large, we mainly discuss interesting models satisfying \ref{A5} for $d=1$ and $m=2$.\\

\textit{\textbf{\textup{(i)} Dispersive perturbations of the Burger's equation.}}\\
In our first example, we discuss the case of some dispersive perturbations of the Burger's equation as in \cite{MV} (of which we borrowed the terminology of strong non-resonance). Namely, we look at \eqref{EQ} with $d=1$, $m=2$, and $\NN(u)=\dx(u^2)$ so that \ref{A1}-\ref{A3}-\ref{A4}$_\beta$ are satisfied with $\beta=(1,0,0)$. The symbol $\varphi$ of $\L$ in \ref{A1} is assumed to be odd, to belong to $C^1(\R)\cap C^2(\R^*)$, and to be of order $\alpha+1>1$: for $\xi\in\R$ large enough,
\begin{align}\label{phialpha}
|\varphi'(\xi)|\sim |\xi|^{\alpha}\text{ and }|\varphi''(\xi)|\sim|\xi|^{\alpha-1}.
\end{align}
In this case, \cite[Lemma 2.1]{MV} ensures that \ref{A5}$_\alpha$ is satisfied with $\alpha_1=\alpha$ and $\alpha_2=1$, provided that the initial datum has zero mean value to avoid trivial resonances. This condition is preserved by both the flow of \eqref{EQ:deterministic}, and the modulated equation \eqref{EQ}. While the computation is usual for the deterministic one, let us justify it for the noisy equation: from the Duhamel formula \eqref{Duhamel}, one has
\begin{align*}
\int_{\T^d}u(t)dx= \int_{\T^d}e^{-(W_t-W_0)\L}u_0dx - \int_{\T^d}\dx\Big(\int_0^te^{-(W_t-W_\tau)\L}(u^2)d\tau\Big)dx.
\end{align*}
Since from the well-posedness theory, the time integral in the second term above makes sense in $C([0;T);H^{s+1}(\T^d))\cap X^{s+1}_T$ and is a total derivative, therefore the second term vanishes. As for the first term, the assumptions on $\varphi$ ensure that $\varphi(0)=0$, and thus
\begin{align*}
\int_{\T^d}e^{-(W_t-W_0)\L}u(0)dx = e^{-i(W_t-W_0)\varphi(0)}\widehat{u}_0(0) = \widehat{u}_0(0) = \int_{\T^d}u(0)dx.
\end{align*}
This shows the conservation of the spatial mean $\int_{\T^d}u(t,x)dx$ under \eqref{EQ}. Thus, if we start from an initial data  in
\begin{align*}
H^s_0(\T)=\Big\{u\in H^s(\T),~~\int_{\T^d}u(x)dx = 0\Big\},
\end{align*}
then the linear evolution and the nonlinearity preserve this space, and thus one can apply Theorem~\ref{THM:LWP-NR} to get well-posedness of these models in $H^s_0(\T)$.

Operators $\L$ whose symbol satisfies \eqref{phialpha} can be found in the following physically interesting models.
\medskip
\begin{enumerate}[label=(\alph*)]
\item \textbf{The pure dispersion operator $\L=\dx D^\alpha$ for any $\alpha>0$.}\\

Here $D^\alpha$ is the Fourier multiplier with symbol $|k|^\alpha$. In this case, \eqref{EQ:deterministic} corresponds to the fractional KdV/BO equation. This model arises in the modeling of one-way propagation asymptotics, in some appropriate regime, of hydrodynamics; see e.g. \cite{ABFS,KleinSaut,Lannes} and references therein. This indeed includes as a particular case the KdV equation ($\alpha=2$) and the BO equation ($\alpha=1$). The Cauchy problem for the noiseless periodic fractional KdV/BO equation \eqref{EQ:deterministic} has been widely studied, see e.g. \cite{BO2,KPV,CKSTT,Gubinelli,KappelerTopalov,KillipVisan} in the case of KdV, \cite{Molinet,GKT,KillipLaurensVisan} in the case of BO, and \cite{MV,Schippa,Jockel,Said1,Said2} in the fractional case,  for a non-exhaustive account on this problem.

For this model, for any $\alpha>0$ and $s\in\R$, the assumption \ref{Arho}$_{\alpha,\beta,s}$ reads
\begin{align}\label{rhoalpha}
\begin{cases}
\rho\ge \frac1\alpha,\\
\rho\ge -\frac{2s}{\alpha},\\
\rho >\frac1{\alpha+1}(\frac32-s).
\end{cases}
\end{align}
In particular, specializing the above to $\alpha=1,2$, we find the same restrictions on $\rho$ for the KdV ($\alpha=2$) and BO ($\alpha=1$) cases as in \cite{CGLLO}. However, note that we only build local solutions for these models in Theorem~\ref{THM:LWP-NR}, while \cite{CGLLO} also addresses the global well-posedness, nonlinear smoothing and Galerkin approximation, invariance of the white noise measure, and stochastic perturbations. See also Remark~\ref{REM:global} below.

While KdV and BO have a regular flow map at high regularity, namely, $s_{\mathrm{reg}}^{\mathrm{KdV}}=-\frac12$ \cite{BO97} and $s_{\mathrm{reg}}^{\mathrm{BO}}=0$ \cite{MST}, when $0<\alpha<1$, it is known \cite{Said1,Said2} that the flow map of the deterministic fractional KdV/BO equation \eqref{EQ:deterministic} is \emph{never} regular on $H^s_0(\T)$ (namely, $s_{\mathrm{reg}}=\infty$). Thus in this case, we witness the second regularization effect stated in Theorem~\ref{THM:LWP-NR}~(iii), namely, strong regularization by noise effect at the level of the regularity of the flow map, since for \eqref{EQ} one finds $s_{\mathrm{reg}}\le s<\infty$ for any $s\in\R$ as soon as $\rho$ satisfies \ref{A5}$_{\alpha,\beta,s}$ discussed above. 

\medskip
\item \textbf{Other one-dimensional perturbations of Burgers on $\T$}.\\

Following \cite[Remark 1.1]{MV}, the conditions \eqref{phialpha} are also satisfied for the linear operator $\L$ of the following equations:
\begin{itemize}
\item \textbf{The ILW equation with fixed depth:} in this case $\L=\dx D \coth(D)$ and satisfies \eqref{phialpha} with again $\alpha=1$. Thus we find the same regime for $\rho$ as for BO ensuring the well-posedness in $H^s_0(\T)$ for this model, similarly as in \cite{CGLLO};
\item \textbf{The Smith operator $\L=\dx (1-\dx^2)^\frac12$:} again, it satisfies \eqref{phialpha} with $\alpha=1$, and the same result as above holds.
\end{itemize}
\end{enumerate}

\begin{remark}\label{REM:ZeroDispersion}
\rm

As we see, if the noise is irregular enough (depending on $\alpha$ and $s$), as soon as one has some dispersion added to the Burger's equation $(\alpha>0$), well-posedness with a regular flow map holds. However, when $\alpha=0$, one finds the inviscid Burger's equation with multiplicative noise 
\begin{align}\label{Burger}
\dt u - \frac{d W_t}{dt}\dx u +\dx(u^2)=0.
\end{align}
It turns out that this equation is not well-posed in the classical sense (i.e. not looking at entropy solutions); see \cite[Section 5.3.1]{Flandoli}. Thus the noise does not regularize the dynamics in the limit case $\alpha=0$, which is consistent with the lower bound on $\rho$ in \eqref{rhoalpha} which blows-up as $\alpha\searrow 0$.
\end{remark}

\begin{remark}\label{REM:ZeroMean}
\rm 

Regarding the zero-mean condition in the results discussed above for dispersive perturbations of the Burger's equation, similarly as for the deterministic equation \eqref{EQ:deterministic}, one can define a flow for initial data of any mean-value by the formula
\begin{align*}
\Psi^t(u):= (\TT_{\widehat{u}_0}^t)^{-1}\circ\Psi^t_0\circ \TT_{\widehat{u}_0}^t(u),
\end{align*}
where $\Psi^t_0$ is the flow on $H^s_0(\T)$, and
\begin{align*}
\TT_c^t : u\mapsto e^{-ct\dx}u-c.
\end{align*}
This allows to extend the flow map $\Psi_0^t$ defined previously on $H^s_0(\T)$, to the whole space $H^s(\T)$. However, note that in doing so, one loses the Lipschitz regularity of the flow map, since $\TT_c^t$ is not Lipschitz continuous with respect to both $u$ and $c$.

\end{remark}

 \textit{\textbf{\textup{(ii)} Non-resonant nonlinear Schr\"odinger (NLS) equation.}}
Following Kenig-Ponce-Vega \cite{KPVNLS}, our last example is the NLS equation for some non-gauge invariant nonlinearity, namely \eqref{EQ} with $\L=-i\Delta$ and $\NN(u)=\cj{u}^m$ for $m\ge 2$, in any dimension $d\ge 1$. Then it is straightforward to check that in this case the resonant function \eqref{resonance} is given by
\begin{align*}
\Phi_{\L,\NN}(k_0,\dots,k_m) = \sum_{j=0}^m|k_j|^2.
\end{align*}
Therefore \ref{A4}$_\beta$ is satisfied with $\beta=0$ and\footnote{Strictly speaking, when $|k_0^*|=0$, \eqref{NR} fails, but this gives the decomposition $\NN=\NN_0+\S$, where $\NN_0$ satisfies \ref{A3}-\ref{A4}$_0$-\ref{A5}$_2$, and $\S$ is infinitely smoothing.} \ref{A5}$_\alpha$ with $\alpha=2$. Thus Theorem~\ref{THM:LWP-NR} applies to the modulated NLS with $\cj{u}^m$ nonlinearity, giving well-posedness in $H^s(\T^d)$ provided that $\rho$ satisfies
\begin{align*}
\begin{cases}
\rho\ge -s,\\
\rho> \frac{m-2}{2}\max(\frac{d}2-s;0).
\end{cases}
\end{align*}

\medskip
\subsubsection{\textbf{Strongly non-resonant/completely resonant case}}

We now turn to the case of strongly non-resonant PDEs perturbed by completely resonant nonlinearities. To state our result, we also define the following parameters: given $s\in\R$, $m\in\N^*$, and $\beta_\RR\in\R_+^{m+1}$, we take
\begin{align}\label{nu}
\nu=\nu(s)>\max\big(|\beta_\RR|-(m-1)s;0\big),
\end{align}
and we define the following strengthened version of \ref{Arho}.
\begin{assumption}{\textup{\textbf{(A$_\rho$*)}}}\label{Arho*} (Enough irregularity of $W_t$)
\begin{align*}
\begin{cases}
\rho\ge  \alpha_1^{-1}\big(\beta_0+\beta_1^*+|\beta_\RR|+(2-\gamma)\nu\big),\\
\rho\ge  \alpha_1^{-1}\big(\beta_1^*+\beta_2^*-2s+(1-\gamma)\nu),\\
{\displaystyle \rho>(\alpha_1+\alpha_2)^{-1}\big(\beta_1^*+\beta_2^*-2s+\sum_{\ell=2}^{L}(\beta_{\ell+1}^*+\frac{d}2-s)+(1-\gamma)\nu\big),}\\
\qquad\qquad\qquad\qquad \text{ for any } 2\le \ell_0\le m\text{ and }2\le L\le \ell_0-1,\\
{\displaystyle \rho>(\alpha_1+\alpha_2)^{-1}\big(\beta_1^*+\beta_2^*+\beta_0+|\beta_\RR|-s+\sum_{\ell=2}^{L}(\beta_{\ell+1}^*+\frac{d}2-s)+(2-\gamma)\nu\big)},\\
\qquad\qquad\qquad\qquad\text{ for any }2\le \ell_0\le m\text{ and }\ell_0+1\le L\le m,
\end{cases}
\end{align*}
\end{assumption}
\begin{theorem}\label{THM:LWP-R}
Assume that {\ref{A0}-\ref{A1}-\ref{A5*}$_{\alpha,\beta,\beta_\RR}$} hold for some $\gamma\in(\frac12;1)$ and $\rho>0$. Then \eqref{EQ} is locally well-posed in $H^s(\T^d)$ for any $s\in\R$ provided that $\rho$ satisfies \ref{Arho*}$_{\alpha,\beta,\beta_\RR,s}$. More precisely:\\
\textup{(i)} \textup{(Well-posedness at high regularity)} If $\sigma\ge s_\RR=\max(s;\frac{|\beta_\RR|}{m-1})$, for any $u_0\in H^\sigma(\T^d)$, there is $T=T(\|u_0\|_{H^{s_\RR}})>0$ and a unique solution $u\in C([0;T);H^\sigma(\T^d))\cap \X^\sigma_T$ to \eqref{EQ} with $u_{t=0}=u_0$. Moreover, the flow map $\Psi : u_0\in H^\sigma(\T^d) \mapsto C([0;T);H^\sigma(\T^d))\cap\X^\sigma_T$ is locally Lipschitz continuous;\\
\textup{(ii)} \textup{(Existence at low regularity)} For any $u_0\in H^s(\T^d)$, there exists $T=T(\|u_0\|_{H^s})>0$ and a mild solution $u\in C([0;T);H^{s}(\T^d))\cap \Fb^{s,\nu}_T\cap \Eb^s_T$ to \eqref{EQ} satisfying $u_{t=0}=u_0$;\\
\textup{(iii)} \textup{(Uniqueness)} This solution is unique in $C([0;T);H^{s}(\T^d))\cap \Fb^{s,\nu}_T\cap \Eb^s_T$;\\
\textup{(iv)} \textup{(Continuity of the flow)} The flow map $\Psi : u_0\in H^{s}(\T^d) \mapsto u\in C([0;T);H^{s}(\T^d))\cap \Fb^{s,\nu}_T\cap \Eb^s_T$ is the unique continuous extension of the mapping $\Phi_T : H^\infty(\T^d)\to C([0;T);H^s(\T^d))\cap \Fb^{s,\nu}_T\cap \Eb^s_T$.
%\textup{(v)} \textup{(Globalization)} If moreover \ref{A2} also holds and $\rho$ satisfies ??????, then for any $s\ge 0$ the local solution can be extended globally in time.
\end{theorem}

In contrast to Theorem~\ref{THM:LWP-NR}, Theorem~\ref{THM:LWP-R} only claims continuity of the flow map. This is because in this case, the solution is constructed through an energy/compactness method. With the assumption in \ref{A5*}$_{\alpha,\beta,\beta_\RR}$ on the completely resonant part, we first show local well-posedness for smooth enough data $u_0\in H^{s_\RR}(\T^d)$ via the same argument as for Theorem~\ref{THM:LWP-NR}. Then a priori energy estimates with the assumption on the resonant part of the nonlinearity in \ref{A5*}$_{\alpha,\beta,\beta_\RR}$ imply that this part does not contribute to the growth of the energy, while the strongly non-resonant assumption of the remaining part of the nonlinearity allows for an a priori estimate only involving the $H^s$ norm of the solution, for $s$ as in Theorem~\ref{THM:LWP-R}. This allows to finally show that for any smooth approximations of a data $u_0\in H^s(\T^d)$, a priori bounds allow to get compactness of the approximations and the existence part of Theorem~\ref{THM:LWP-R}~(ii) follows. For the uniqueness, the resonant part does not vanish anymore due to the symmetry breaking for the difference equation. Thus we follow an idea of \cite{OW} and show uniqueness for solutions sharing the same initial data, since we can estimate the difference of the resonant parts for such solutions. The continuity of the flow map follows from straightforward modifications of the arguments. Note however that to have a nonlinear estimate controlling the resonant part when the regularity $s$ is negative, we resort to frequency-dependent short-time refinements $\Fb^{s,\nu}_T$ of the $U^2$ type spaces $\X^s_T$. The energy space $\Eb^s_T$ is then a logarithmic strengthening of $C([0;T);H^s(\T^d))$ adapted to the short-time structure, see Definition~\ref{DEF:F} below. 

Let us also point out that, compared to \ref{A5}$_{\alpha,\beta,s}$ assumed to treat semilinear problems, in Theorem~\ref{THM:LWP-R} the assumption \ref{A5*}$_{\alpha,\beta,\beta_\RR,s}$ on $\rho$ now depends on $\gamma$. This is a consequence of the energy method via short-time adapted spaces that we use to deal with the completely resonant term. This observation is better exploited in \cite{ModulatedKP} where we study the modulated KP-I equation, which is a prototype of quasi-linear dispersive PDE for which one has to resort to refined energy methods since fixed-point arguments fail due to a large set of resonant interactions preventing the flow map from being regular.

\subsubsection{\textbf{Examples of completely resonant perturbations}}\label{EX:R}
We give below some examples of equations satisfying the assumptions of Theorem~\ref{THM:LWP-R}.

\medskip
\textit{\textbf{\textup{(i)} Fractional Wick-ordered mKdV equation.}}\\
Our first example is given by \eqref{EQ} with again $d=1$, $\L=\dx D^\alpha$, but now $m=3$ and 
$$\NN(u)= \big(u^2-\frac1{2\pi}\int_\T u^2(x)dx\big)\dx u$$
is the Wick-ordered cubic nonlinearity of the modified KdV equation. In the deterministic case \eqref{EQ:deterministic}, and with the usual cubic nonlinearity $\NN(u)=u^2\dx u$, this model describes the weakly nonlinear propagation of long waves in shallow channels with various dispersion effect. It reduces to the mKdV equation when $\alpha=2$, and to the modified BO equation when $\alpha=1$. This model has also been widely studied: see e.g. \cite{BO2,BO97,TT,NTT,MPV,KappelerTopalov,GLM,Chapouto1,Chapouto2} and references therein.

The $L^2$ norm is conserved for the deterministic equation \eqref{EQ:deterministic}, and from an argument similar as the one presented in the proof of Lemma~\ref{LEM:energy} below, one can check that it is also conserved for \eqref{EQ}, by a direct computation for spatially smooth solutions at the level of the interaction representation formulation of the equation, and through an approximation argument for rougher solutions. Then, if the data is at least in $L^2(\T)$, one can use the gauge transformation
$$U(t,x)=u\big(t,x-\frac{t}{2\pi}\int_\T u_0(x)^2dx\big)$$
to convert solutions of the mKdV to solutions of the Wick-ordered mKdV. However, this transformation becomes ill-defined for data rougher than $L^2$, which is included in the case we consider here. Thus we focus on the analysis of the Wick-ordered equation. Note that by an argument similar to the one in \cite{GO,OW}, a byproduct of Theorem~\ref{THM:LWP-R} applied to this model is that there exist \emph{no} solution to the original modulated periodic mKdV equation below $L^2$.
 
Then the Wick-ordered nonlinearity can be decomposed (see \cite{BO2}) as
$$\NN(u)=\big(u^2-\frac1{2\pi}\int_\T u^2(x)dx\big)\dx u = \NN_0(u)+\RR(u),$$
where
\begin{align}\label{NmKdV}
\NN_0(u)(x)=\frac13\sum_{k_0\in\Z}e^{ik_0x}ik_0\sum_{\substack{k_1,k_2,k_3\in\Z\\k_1+k_2+k_3=k_0\\(k_1+k_2)(k_1+k_3)(k_2+k_3)\neq 0}}\widehat{u}_{k_1}\widehat{u}_{k_2}\widehat{u}_{k_3},
\end{align}
and 
\begin{align}\label{RmKdV}
\RR(u)(x)=-\sum_{k\in\Z}e^{ikx}ik|\widehat{u}_k|^2\widehat{u}_k.
\end{align}
Clearly, $\RR$ satisfies \eqref{NR} with $\beta_\RR=(1,0,0)$. Moreover, it is proved\footnote{Note that $\alpha$ in our context corresponds to $\alpha+1$ in \cite{GH}.} in \cite[Lemma 4.1]{GH} that the resonance function \ref{resonance} satisfies for this model
\begin{align}\label{NRmKdV}
\big|\Phi_{\L,\NN}(k_0,k_1,k_2,k_3)\big|&=\Big|k_1|k_1|^\alpha+k_2|k_2|^\alpha+k_3|k_3|^\alpha-k_0|k_0|^\alpha\Big|\notag\\
& \sim |k_0^*|^{\alpha-1}|(k_1+k_2)(k_1+k_3)|\gtrsim |k_0^*|^{\alpha-1}\sim\jb{k_0^*}^{\alpha-1}
\end{align}
on the support of $\NN_0$. Thus $\NN_0$ satisfies \ref{A4}-\ref{A5} with $\alpha_1=\alpha-1$, $\alpha_2=0$, and $\beta=(1,0,0)$, as soon as $\alpha>1$, i.e. for dispersion strictly greater than that of the modified BO equation. This includes as a particular example the case of the modulated mKdV equation ($\alpha=2$) treated in \cite{CG2,CGLLO}. While \cite{CG2,CGLLO} proved that the modulated mKdV equation is well-posed in $H^s(\T)$ for any $s\ge \frac12$ under the mere condition $\rho \ge \frac12$, in the present work Theorem~\ref{THM:LWP-R} allows to deal with any regularity $s\in\R$, in particular $s<\frac12$, provided that
\begin{align*}
\begin{cases}
\rho> 2+(2-\gamma)(1-2s) \text{ if }\frac15\le s <\frac12,\\
\rho > 2+(2-\gamma)(1-2s) + 1-5s \text{ if }s\le \frac15. 
\end{cases}
\end{align*}

At last, note that as $s\nearrow \frac12$ and $\gamma\searrow\frac12$, we find the condition $\rho>2$, which is much worse than the condition $\rho\ge \frac12$ imposed in \cite{CG2,CGLLO} to treat the regime $s\ge \frac12$. This is partly because we use a short-time energy method, but even Lemma~\ref{LEM:LWP} gives the condition $\rho\ge 1$ when studying the well-posedness for higher regularity $s\ge s_\RR=\frac12$. The main reason is that a more careful case by case analysis shows that the lower bound in \eqref{NRmKdV} can actually be improved to $\alpha_2=1$, except in the case $|k_0|\sim|k_1|\sim|k_2|\sim|k_3|$. Thus it may be possible to refine the analysis performed in the proofs of Lemma~\ref{LEM:LWP} and Theorem~\ref{THM:LWP-R} in this particular case to take account of this fact and improve the condition \ref{Arho*} on $\rho$.

\medskip
\textit{\textbf{(ii) Fifth-order KdV equation.}}\\
Our next example is given by the fifth-order KdV equation, which is the second equation after KdV in the KdV hierarchy of completely integrable models. It is the Hamiltonian equation associated with the third conservation law of KdV, the first ones being (cf. \cite{BO2})
$$E_0(u)=\int_\T u^2dx,$$
$$E_1(u)=\int_\T(\dx u)^2dx-\frac13\int_\T u^3dx,$$
and $$E_2(u)=\int_\T(\dx^2 u)^2dx-\frac53\int_\T u(\dx u)^2dx+\frac5{36}\int_\T u^4dx.$$
Together with the Poisson structure of the KdV hierarchy, \eqref{EQ:deterministic} reads in this case (neglecting constants)
\begin{align}\label{5KdV}
\dt u =\dx \frac{dE_2}{du}(u) = \dx^5u + \dx u \dx^2u +\dx(u\dx^2 u) +u^2\dx u 
\end{align}
Thus after the same renormalizations performed for KdV in Subsubsection~\ref{EX:NR} and mKdV in the previous example, one can study the modulated version of \eqref{5KdV} with a Wick-ordered cubic nonlinearity, and in the Sobolev space $H^s_0(\T)$ of functions with zero spatial mean. This yields the model equation
\begin{align*}
\dt u - \dx^5 u +\NN_1(u)+\NN_2(u)+\NN_3(u)+\RR(u)=0,
\end{align*} 
where the quadratic nonlinearities $\NN_1,\NN_2$ are associated with the resonance function of the fractional KdV equation (with $\alpha=4$). Thus $\NN_1$ satisfies \ref{A5} with $\alpha_1=\alpha=4$ and $\alpha_2=1$, and \ref{A4} with $\beta=(0,1,2)$, while $\beta=(1,0,2)$ for $\NN_2$. The last terms $\NN_3+\RR$ correspond to the Wick-order cubic nonlinearity as in \eqref{NmKdV}-\eqref{RmKdV}. Thus this model also fits into the framework of Theorem~\ref{THM:LWP-R}.

\medskip
\textit{\textbf{(iii) Fractional Wick-ordered cubic NLS.}}\\
Our last example deals with the Wick-ordered fractional cubic NLS, i.e. $d=1$, $m=3$, $\L=iD^\alpha$ for some $\alpha>0$, and $\NN(u)=i(|u|^2-\int_\T |u|^2dx)u$ for complex-valued functions $u$. Again, we have the decomposition
$$\NN=\NN_0+\RR,$$
where
$$\NN_0(u)=i\sum_{k_0\in\Z}e^{ik_0x}\sum_{k\substack{k_1,k_2,k_3\in\Z\\k_1-k_2+k_3-k_0=0\\
k_2\not\in\{k_1,k_3\}}}\ft u_{k_1}\cj{\ft u}_{k_2}\ft u_{k_3},$$
and 
$$\RR(u)=i\sum_{k\in\Z}e^{ikx}|\ft u_k|^2\ft u_k.$$
The resonance function enjoys in this case the lower bound (see e.g. \cite[Lemma 2.1]{BLLZ})
\begin{align}\label{ResNLS}
|\Phi_{\L,\NN}(k_0,k_1,k_2,k_3)|=\big||k_1|^\alpha -|k_2|^\alpha+|k_3|^\alpha-|k_4|^\alpha\big|\gtrsim |(k_1-k_2)(k_3-k_3)||k_0^*|^{\alpha-2}
\end{align}
for any $\alpha>1$. Thus we see that for $\alpha>2$, this is bounded below as required on the support of $\NN_0$, so that this model satisfies \ref{A5*}$_{\alpha,\beta,\beta_\RR}$ with $\alpha_1=\alpha-2$, $\alpha_2=0$ and $\beta=(0,0,0,0)=\beta_\RR$. Actually, a more refined analysis as for the previous example gives that \eqref{ResNLS} can be improved to $\alpha_1=\alpha-1$, except for the High-$\dots$-High to High interactions $|k_0|\sim|k_1|\sim |k_2|\sim|k_3|$ where \eqref{ResNLS} is sharp. We see that the case of the usual Schr\"odinger equation ($\alpha=2$) is thus just out of reach of our method. As the cubic NLS is one of the most prominent example of nonlinear dispersive PDE, it would be of interest to investigate further if we can improve the well-posedness theory for the modulated equation \eqref{EQ} compared to \eqref{EQ:deterministic}; see also \cite[Theorem 1.13]{ModulatedNLS}.

\begin{remark}
\rm
To the best of our knowledge, the currently best well-posedness result for the fractional Wick-ordered cubic NLS available in the literature is the recent work \cite{BLLZ}, where the authors used a second gauge transform in the spirit of \cite{OTW} and of the modified $X^{s,b}$ spaces of \cite{TT,NTT}. However, this approach leads to a perturbation of the resonance function $\Phi_{\L,\NN}$ in \eqref{ResNLS} by $\wt\Phi = |\widehat{(u_0)}_{k_0}|^2-|\widehat{(u_0)}_{k_1}|^2+|\widehat{(u_0)}_{k_2}|^2-|\widehat{(u_0)}_{k_3}|^2$. While the latter is indeed a perturbation of the former when $u_0\in H^s(\T)$ with $s<\frac{2-\alpha}{2}$ in view of \eqref{ResNLS}, in the regime of $s$ we consider this argument seems to distort too much the resonance relation \eqref{ResNLS} and thus we find that the short-time energy method is more efficient in our case.
\end{remark}

\subsection{Sketch of the argument}\label{SUBS:proof}
After having given some examples of usual dispersive PDEs satisfying the set of assumptions of Theorem~\ref{THM:LWP-NR}, let us briefly sketch the argument behind the regularization phenomenon appearing in Theorem~\ref{THM:LWP-NR} following the idea introduced by Chouk and Gubinelli \cite{CG2}, and discuss the modifications given in our argument.

The local solution space $\X^s_T$ in Theorem~\ref{THM:LWP-NR} is a $U^2$-type space adapted to the modulated linear evolution $e^{-(W_t-W_0)\L}$; see Subsection~\ref{SUBS:space} below for definitions and properties of these spaces. It replaces the space $D^W(H^s)$ used in \cite{CG1,CG2} and further developed in \cite{CGLLO}. In particular the duality relation of Proposition~\ref{PROP:UpVp}~(iii) below can be seen at the endpoint case $\gamma=\frac12$ of the nonlinear Young integral theory of \cite[Theorem 2.3]{CG1} for $f\in C^\gamma_t\mathrm{Lip}_M(H^s)$ and $g\in C^\gamma_tH^s$.
%Thus Theorem~\ref{THM:LWP-NR}~(i) is a \textit{conditional} well-posedness result, since uniqueness of the solution $u\in C([0;T);H^s(\T^d))$ only holds if the solution also belongs to $\X^s([0;T))$. 

Indeed, the proof of Theorem~\ref{THM:LWP-NR} relies on a fixed point argument in the solution space $C([0;T);H^s(\T^d))\cap \X^s_T$ for the Duhamel formulation 
\begin{align}\label{Duhamel}
u(t)=e^{-(W_t-W_0)\L}u_0-\int_0^te^{-(W_t-W_\tau)\L}\NN(u(\tau))d\tau,
\end{align}
of \eqref{EQ}. It has been known since the early pioneering works of Bourgain \cite{BO1,BO2} (see also \cite{Beals,RR}) that instead of building a mild solution $u\in C([0;T);H^s(\T^d))$ to \eqref{EQ:deterministic} by working directly on \eqref{Duhamel} (when $W_t=t$), one can better exploit the oscillations related to the dispersive nature of the equation by passing to the so-called interaction representation of \eqref{Duhamel}; namely, looking at the equation for $v(t)=e^{t\L}u(t)$:
\begin{align}\label{interaction}
v(t)= v(0)-\int_0^te^{\tau\L}\NN\big(e^{-\tau\L}u(\tau)\big)d\tau.
\end{align}
Indeed, writing the integral equation \eqref{interaction} on the Fourier side, this yields
\begin{multline*}
\widehat{v}_{k_0}(t)= \widehat{v}_{k_0}(0)-\sum_{\substack{k_1,\dots,k_m\in \Z^d\\ \sum_{j=0}^m\pm_j k_j=0}}\widehat{\NN}(k_0,\dots,k_m)\int_0^te^{-i\tau\Phi_{\L,\NN}(k_0,\dots,k_m)}\prod_{j=0}^{m}\widehat{u}_{k_j}^{\pm_j}(\tau)d\tau,
\end{multline*}
where the phase function $\Phi_{\L,\NN}$ is defined in \eqref{resonance}. Then one can exploit the oscillatory factor in the time integral above to deal with the nonlinearity. The $X^{s,b}$ approach popularised by Bourgain in \cite{BO1,BO2} amounts to look for a solution $v$ to \eqref{interaction} not only continuous in time, but also having some regularity, namely belonging to $H^b_tH^s_x$ for some\footnote{so that $H^b_tH^s_x\subset C_tH^s_x$. In \cite{BO2} and subsequent works, suitable adaptations have been devised to deal with the borderline case $b=\frac12$.} $b>\frac12$. Indeed, this regularity allows to exploit the oscillations from the phase function in the above time integral: schematically, from the boundedness of the time integral and a standard product rule in Sobolev spaces, we can estimate
\begin{align*}
&\Big\|\int_0^te^{-i\tau\Phi_{\L,\NN}(k_0,\dots,k_m)}\prod_{j=0}^{m}\widehat{u}_{k_j}^{\pm_j}(\tau)d\tau\Big\|_{H^b_t}\les \big\|e^{-it\Phi_{\L,\NN}(k_0,\dots,k_m)}\prod_{j=0}^{m}\widehat{u}_{k_j}^{\pm_j}\big\|_{H^{b-1}_t}\\
&\qquad\les \big\|e^{-it\Phi_{\L,\NN}(k_0,\dots,k_m)}\big\|_{W^{b-1,\infty}_t}\|\widehat{u}_{k_0}\|_{H^{1-b}_t}\prod_{j=1}^{m}\|\widehat{u}_{k_j}\|_{H^b_t}\sim \jb{\Phi}^{b-1}\|\widehat{u}_{k_0}\|_{H^{1-b}_t}\prod_{j=1}^{m}\|\widehat{u}_{k_j}\|_{H^b_t}.
\end{align*}
While this analysis is usually performed by working on the Fourier side in time, Koch-Tataru \cite{KT05,KT07} interpreted the time integral above as
$$\int_0^t\Big(\prod_{j=0}^{m}\widehat{u}_{k_j}^{\pm_j}(\tau)\Big)dF(\tau),$$
where ${\displaystyle F(\tau)=\frac{e^{i\tau\Phi}}{i\Phi}}$. To benefit from the multilinear oscillations, one cannot spare too much of regularity of $F$, and thus this integral does not fall into the scope of Stieltjes integration, but is in the Young regime instead. In order to treat nonlinear dispersive PDEs at scaling critical regularity, one needs the ``endpoint" case of the Young integral (similar to the endpoint case $b=\frac12$ in $X^{s,b}$ analysis), corresponding to the duality between $DU^2$ and $V^2$; see \cite{KTV} and Proposition~\ref{PROP:duality} below.

In this work, we thus rely on the properties of $U^p/V^p$ spaces to deal with the time integral
$$\int_0^te^{-iW_\tau\Phi_{\L,\NN}(k_0,\dots,k_m)}\prod_{j=0}^{m}\widehat{u}_{k_j}^{\pm_j}(\tau)d\tau$$
when $W_t$ is rough. Indeed, the atomic structure of these spaces (see Definition~\ref{DEF:UpVp} below) is particularly adapted to exploit the $(\rho,\gamma)$-irregularity of $\mu$ \eqref{mu} through the occupation time formula \eqref{occupation}; see the proof of our main multilinear estimate in Lemma~\ref{LEM:Multi} below. While adapted $U^p/V^p$ spaces are nowadays standard spaces to deal with semilinear dispersive PDEs at scaling critical regularity thanks to the sharp estimation of time integrals as above, we emphasize here that we heavily rely on their atomic structure to deal with the \emph{sub}-critical\footnote{Again, because of the change of scaling induced by the noise in \eqref{EQ}, the regime of regularity treated in Theorems~\ref{THM:LWP-NR} and~\ref{THM:LWP-R} is sub-critical with respect to this new scaling, which is translated in assumptions \ref{Arho} and \ref{Arho*}.} equations treated in this work. Although it would in principle be possible to rely on the more standard $X^{s,b}$ spaces, this would require to work only on the physical side to deal with the $H^b_t$ norm of the time integral above, which differs from the systematic use of the space-time Fourier transform analysis performed in $X^{s,b}$ spaces (see e.g. \cite{Tao}). 

In the setting of rough stochastic differential equations, the use of Young integrals is much more usual than the $U^p/V^p$ spaces, even though the latter have been used in the context of rough paths theory \cite{FrizVictoir}. This approach amounts to choosing the topology of H\"older spaces to measure the time regularity of the functions appearing in the time integral above, i.e. the time regularity of the solution to the interaction representation formulation. In the context of deterministic nonlinear dispersive PDEs, this has been advertised in the work of Gubinelli \cite{Gubinelli} on the periodic KdV equation, following the development of controlled rough paths for rough SDEs in \cite{GubinelliJFA}. Regarding the case of modulated equations, this is also the approach followed in \cite{CG1,CG2} and extended in \cite{CGLLO}; see also \cite{CLO,COZ} for other stochastic dispersive PDEs with multiplicative noise. In the context of semilinear PDEs covered by Theorem~\ref{THM:LWP-NR}, both approaches seem to be equally efficient. In the case of quasilinear PDE without modulation, short-time $U^p/V^p$ spaces have become standard, which is why we rely on this technology to deal with the modulation in Theorem~\ref{THM:LWP-R}; see also \cite{ModulatedKP}. Finally, let us refer to the review \cite{GaReview} for a detailed account on the (nonlinear) Young integration theory heavily used in the context of rough SDEs. 

\medskip
We conclude this introduction with some further remarks.

\begin{remark}\label{REM:global}
\rm
Even if the corresponding deterministic equation \eqref{EQ:deterministic} is Hamiltonian with a conserved energy, the latter is not conserved anymore for \eqref{EQ}. However, if the $L^2$ norm is conserved for \eqref{EQ:deterministic}, one can check in practice that it is also conserved for \eqref{EQ}, as mentioned in our discussion for the modulated mKdV equation above. This is the case for many interesting models, including most of the ones presented in Subsubsections~\ref{EX:NR} and~\ref{EX:R} above. Thus, under assumption \ref{Arho}$_{\alpha,\beta,0}$ or \ref{Arho*}$_{\alpha,\beta,\beta_\RR,0}$ ensuring well-posedness in $L^2(\T^d)$, one can iterate the local well-posedness statements of Theorems~\ref{THM:LWP-NR} or~\ref{THM:LWP-R} to globalize the solutions. Following \cite[Theorem 1.18 and Theorem 1.19]{CGLLO}, for these same models, provided that we take now $s<-\frac{d}2$ in the assumptions of Theorem~\ref{THM:LWP-NR} and~\ref{THM:LWP-R}, we could also use Bourgain's invariant measure argument \cite{BO94,BO96} to show that the white noise measure
\begin{align*}
d\nu = ``\ZZ^{-1}e^{-\frac12\|u\|_{L^2}^2}du",
\end{align*}
formally defined as the Gaussian measure on ${\displaystyle \bigcap_{s<-\frac{d}2}H^s(\T^d)}$ with the identity as covariance operator, is invariant under \eqref{EQ} and can thus be used as a ``statistical'' conservation law to globalize the local solution starting from $\nu-$a.e. initial data $u_0\in H^s(\T^d)$, $s<-\frac{d}2$.
\end{remark}

%\begin{remark}
%\rm
%In assumption \ref{A5*}, we only assume that the integrable part $\ft \RR(u)_{k_0}$ contains only decoupled terms in the actions, i.e. no terms like $\sum_{k_1}|\ft u_{k_0}|^{2m_0}\ft u_{k_0}|\ft u_{k_1}|^{2m_1}$ for some $m_1\in\N^*$. This is because this term has no multilinear smoothing at all, and thus in the regime $|k_1|\gg |k_0|$, one cannot control the contributions of the $k_1$ frequencies by using the localization on small frequency-dependent time intervals.
%\end{remark}

\begin{remark}
\rm
As we discussed in the introduction, we only consider the case of \eqref{EQ} on (square) tori. Indeed, in other geometries such as the Euclidean space or irrational tori, one can have bad High$\times\dots\times$High$\to$Very Low quasi-resonant interactions (i.e. $|\Phi_{\L,\NN}|\les 1$) which are not improved by the noise. As we have seen in the examples above, on square tori one can remove exact resonances using gauge transformations to prevent this issue.
\end{remark}

%\begin{remark}
%????????????Regularity range for well-posedness in Theorem~\ref{THM:LWP-NR} non optimal (e.g. compared to the KdV case treated in \cite{CG2}). This is because we do not perform an analysis tailored to the particular symbol of $\L$ but our point is to emphasize that we can treat super-critical regularities $s$ for \eqref{EQ} compared to the deterministic case \eqref{EQ:deterministic} provided that $\rho$ is large enough depending on $s$. The relation between $(\gamma,\rho)$ and $s$ can be more optimised through a finer, model-dependent analysis, incorporating e.g. linear or multilinear Strichartz estimates, transversality estimates,... See \cite{ModulatedNLS,ModulatedKP} where these kind of estimates are adapted to the modulated equation \eqref{EQ}. 
%\end{remark}

\begin{remark}
\rm

As mentioned in the introduction, one motivation to study the modulated equation \eqref{EQ} comes from nonlinear optics, where the modulated NLS equation on $\M=\R$ appears as an effective model for the propagation of a signal in optical fibre with dispersion management. However, this particular model does not satisfy the strongly non-resonant condition. As a first step towards understanding a potential regularization by noise (in terms of the range of local well-posedness) for \eqref{EQ} in the non strongly non-resonant case, we study in \cite{ModulatedKP} the modulated KP-I equation, which is a model known to have a large set of quasi-resonances, preventing Theorems~\ref{THM:LWP-NR} or~\ref{THM:LWP-R} to apply. We show in this case that the noise can indeed improve the worst quasi-resonant High-Low interaction in a short-time argument.
\end{remark}

\begin{ackno}

\rm 
The author would like to thank Massimiliano Gubinelli, Tadahiro Oh, and Nikolay Tzvetkov for helpful discussions, which in particular improved the presentation of this work. The author was partially supported by the ANR project Smooth ANR-22-CE40-0017 and PEPS JCJC project BFC n°272076.

\end{ackno}

\section{Preliminaries}\label{SEC:Prelim}
\subsection{Notations}
For complex numbers $a\in\C$, we write $a^-=\cj a$ and $a^+=a$. For positive numbers $A$ and $B$, we write $A\les B$ if there is a constant $C>0$, independent of the various parameters, such that $A\le CB$; $C=10^{42}$ should work throughout this text. We also write $A\sim B$ if both $A\les B$ and $B\les A$. We also write $A\vee B = \max(A;B)$ and $A\wedge B=\min(A;B)$.

In the following, we use $N,N_j,M\in 2^{\N}$ to denote dyadic integers in $2^{\N}=\{2^m,~~m\in \N\}$. For $N\in 2^\N$ a dyadic integer, we first define $\P_{\le N}$ to be the Dirichlet projection onto the frequencies $\{|k|\leq N\}\subset \Z^d$, and $\P_N=\P_{\le N}-\P_{\le \frac{N}2}$, with the convention that $\P_{\le \frac12}u=\ft u_0$. %More generally, for $S\subset\Z^d$ a set of frequencies, we define the projector $\P_S$ as the Fourier multiplier with symbol $\mathbf{1}_S(k)$, $k\in\Z^d$.

At last, for a non-negative tuple $\beta\in [0;\infty)^{m+1}$, we write $|\beta|=\sum_{j=0}^m\beta_j$.

\subsection{Properties associated with the noise}
In this whole manuscript, we assume that $W:\R\to\R$ is continuous, and we write $\mu$ for its occupation measure \eqref{mu}. We will always assume that $d\mu_{I}$ is absolutely continuous with respect to the Lebesgue measure $dz$ for any interval $I\subset \R$. We will make heavy use of the following \textbf{occupation time formula}; see \cite[Theorem 6.4]{Review}.
\begin{proposition}
For every non-negative measurable function $F:\R\to [0;\infty)$ and interval $I\subset \R$, it holds
\begin{align}\label{occupation}
\int_IF(W_t)dt = \int_\R F(z)\frac{d\mu_{I}}{dz}(z)dz.
\end{align}
\end{proposition}

\subsection{Function spaces}\label{SUBS:space}

\subsubsection{\textbf{Adapted function spaces}}
In this subsection we recall the definition and properties of $U^2/V^2$-type spaces associated with the flow of the linear dispersive equation with modulated dispersion that we already used in \cite{ModulatedNLS}.

These spaces were first used in \cite{KT05,KT07} to study the Cauchy problem for some dispersive PDEs, and we refer to \cite{KTV,HHK,HTT,SchippaThesis} and in particular \cite{CHT} for the proofs of their properties that we list below. %We give the necessary adaptations to the case of Banach-valued functions as in \cite{KTV} since we consider functions of $t$ with values in $H^s(\T^d)$.

\begin{definition}\label{DEF:UpVp}
\rm
Let $H$ be a Hilbert space, $1\le p<\infty$, and $I=[a;b)\subset \R$ with $-\infty\le a<b\le\infty$. Let then $\mathcal{Z}(I)$ be the collection of finite non-decreasing sequences $\{ t_k \}_{k=0}^K$ in $I$ with $t_0=a$ and $t_K=b$.\\
\textup{(i)} We define $V^p(I)H$ as the space of functions $u: I \to H$ such that $u(t)$ has a limit as $t\searrow a$ and also ${\displaystyle \lim_{t\nearrow b}u(t)=0}$, endowed with the norm
\begin{align*}
\| u \|_{V^p(I)H} = \sup_{ \{t_k \}_{k=0}^K \in \mathcal{Z}(I)}\Big(  \sum_{k=1}^K \|u(t_k) - u(t_{k-1})\|_{H}^p \Big)^{\frac1p}.
\end{align*}
\textup{(ii)} We define a $U^p(I)H$-atom to be a piecewise continuous function $A:I\to H$ such that
\begin{align*}
A(t) = \sum_{k=1}^{K}\ind_{[t_{k-1};t_{k})}\psi_{k}
\end{align*}
for some partition $\{t_{k}\}_{k=0}^K\in\ZZ(I)$ and collection $\{\psi_{k}\}_{k=1}^K\in H^K$ such that ${\displaystyle \sum_{k=1}^K\|\psi_{k}\|_{H}^p\le 1}$. Note that $A$ is right-continuous with ${\displaystyle \lim_{t\searrow a}A(t)=0}$ and has a limit as $t\nearrow b$.

Then $U^p(I)H$ is the space of functions $u:I\to H$ such that
\begin{align*}
u = \sum_{j=0}^\infty \lambda_j A_j,
\end{align*}
with $\{\lambda_j\}_j\in\ell^1(\N;\K)$, and $A_j$ are $U^p(I)H$-atoms, endowed with the norm
\begin{align*}
\|u\|_{U^p(I)H}=\inf\Big\{\sum_{j=0}^\infty|\lambda_j|,~~u=\sum_{j=0}^\infty\lambda_jA_j,~~A_j\text{ are $U^p(I)H$-atoms}\Big\}.
\end{align*}

\end{definition}
These spaces enjoy the following properties; again, we refer to \cite{KTV,HHK,HTT,CHT,SchippaThesis} for the proofs.
\begin{proposition}\label{PROP:UpVp0}
The following properties hold:\\
\textup{(i)} $U^p(I)H$ and $V^p(I)H$ are Banach spaces.\\
\textup{(ii)} If $u\in U^p([a;b))H$, $a<c<d<b$, and $u(c)=0$, then $\mathbf{1}_{[c;d)}(t)u\in U^p([c;d))H$ and $$\|u\|_{U^p([c;d))H}=\|\mathbf{1}_{[c;d)} u\|_{U^p([a;b))H}\le \|u\|_{U^p([c;d))H}.$$\\
Similarly, if $u\in V^p([a;b))H)$, $a<c<d<b$, and $u(d)=0$, then $\mathbf{1}_{[c;d)}(t)u\in V^p([c;d))H$ and $\|u\|_{V^p([c;d))H}+\|u(c)\|_{H}=\|\mathbf{1}_{[c;d)} u\|_{V^p([a;b))H}$.\\
\textup{(iii)} For any $1\le p<q<\infty$, there are the continuous embeddings $
U^p(I)H\subset U^q H\subset L^\infty(I;H)$ and $V^p(I)H\subset V^q(I)H\subset L^\infty(I;H)$.\\
\textup{(iv)} If $u\in U^p(I)H$ with $u(b)=0$, then $u\in V^p(I)H$ and $\|u\|_{V^p(I)H}\les \|u\|_{U^p(I)H}$.\\
Similarly, for any $1\le p<q<\infty$, if $u\in V^p(I)H$ is right-continuous with $u(a)=0$, then $u\in U^q(I)H$ and $\|u\|_{U^q(I)H}\les \|u\|_{V^p(I)H}$.
\end{proposition}

\begin{definition}\label{DEF:DU}
\rm
We define the Banach space $$DU^2(I)H=\big\{\dt u,~~u\in U^2(I)H\big\},$$
where the derivative is taken in the sense of distributions, endowed with the norm $\|f\|_{DU^2(I)H}=\|u\|_{U^2(I)H}$. The condition $u(a)=0$ for $u\in U^p(I)H$ guarantees that for $f\in DU^2(I)H$ there is a unique choice of $u\in U^2(I)H$ such that $f=\dt u$.
\end{definition}

The main result regarding this space is the duality relation between $DU^2(I)H$ and $V^2(I)H$.

\begin{proposition}\label{PROP:duality}
For all $f\in L^1(I;H)\subset DU^2(I)H$, it holds
\begin{align*}
\|f\|_{DU^2(I)H}=\sup_{\|v\|_{V^2(I)H}\le 1}\Big|\int_I\langle f(t),v(t)\rangle
_Hdt\Big|.
\end{align*}
\end{proposition}

\begin{definition}\label{DEF:UpVpW}
\rm
Let $1 \leq p <\infty$. We define the adapted function space
\begin{align*}
U^p_{W\L}(I)L^2(\T^d) = \big\{u:I\to L^2(\T^d) : e^{W_t\L}u\in U^p(I)L^2(\T^d)\big\}
\end{align*}
and 
\begin{align*}
V^p_{W\L}(I)L^2(\T^d) = \big\{ u: I \to L^2(\T^d) : e^{W_t\L} u \in V^p(I) L^2(\T^d) \big\}
\end{align*}
with norms given by
\begin{align*}
\|u\|_{U^p_{W\L}(I)L^2_x}=\|e^{W_t\L}u\|_{U^p(I)L^2_x}\text{ and }\|u\|_{V^p_{W\L}(I)M^2}=\|e^{W_t\L}u\|_{V^p(I)L^2_x}.
\end{align*}
%\end{definition}

%\begin{definition}
%\label{DEF:XY}
%\rm
%Let $s \in \R$ and $r\in[1;\infty]$.\\
%\textup{(i)} We define $X^{s,r}$ as the space of all functions $u:\R\to H^s(\T^d)$ such that for any $k\in \Z^d$, the map $t\mapsto \widehat{\big(e^{W_t\L}u(t)\big)}(k)$ is in $U^2(\R;\K)$ with finite norm
%\begin{align*}
%\|u\|_{X^{s,r}} = \Big\|\langle k\rangle^{s}\big\|\widehat{\big(e^{W_t\L}u(t)\big)}(k)\big\|_{U^2_t}\Big\|_{L^r(\Z^d)}.
%\end{align*}
%
%\noi
%\textup{(ii)} We define $Y^{s,r}$ as the space of all functions $u: \R \times \T^d \to \K$ such that for any $k\in \Z^d$, the map $t\mapsto \widehat{\big(e^{W_t\L}u(t)\big)}(k)$ is in $V^2(\R;\K)$ with finite norm
%\begin{align*}
%\|u\|_{Y^{s,r}} = \Big\|\langle k\rangle^{s}\big\|\widehat{\big(e^{W_t\L}u(t)\big)}(k)\big\|_{V^2_t}\Big\|_{L^r(\Z^d)}.
%\end{align*}
%
%\noi
%\textup{(iii)} For a time interval $I\subset \R$, we define the restricted spaces $X^{s,r}(I)$ and $Y^{s,r}(I)$ by the norms
%\begin{align*}
%\|u\|_{X^{s,r}(I)}=\inf\big\{\|\widetilde{u}\|_{X^{s,r}},~~u\equiv \widetilde{u}\text{ a.e. on }I\big\}\text{and }\|u\|_{Y^{s,r}(I)}=\inf\big\{\|\widetilde{u}\|_{Y^{s,r}},~~u\equiv \widetilde{u}\text{ a.e. on }I\big\}.
%\end{align*}
Following \cite{HTT}, for $s\in\R$, we then define the Banach spaces $\X^s(I)$ and $\Y^s(I)$ via the norms
\begin{align*}
\|u\|_{\X^s(I)}^2=\sum_N N^{2s}\big\|\P_N u\big\|_{U^2_{W\L}(I)L^2_x}^2\qquad\text{and}\qquad\|u\|_{\Y^s(I)}^2=\sum_N N^{2s}\big\|\P_N u\big\|_{V^2_{W\L}(I)L^2_x}^2.
\end{align*}
When $I=[0;T)$ for some $T>0$, we simply write $\X^{s}(I)=\X^s_T$ and $\Y^{s}(I)=\Y^s_T$. 
%We have similar definitions for the restrictions of $U^p_{W\L}H^s(\T^d)$ and $V^p_{W\L}H^s(\T^d)$ to time intervals $I$.
\end{definition}

\begin{remark}
\rm
Since $U^2H^s(\T^d)\subset L^\infty(\R;H^s(\T^d))$ continuously, we have that for any $T>0$, $C([0;T);H^s(\T^d))\cap \X^s_T $ is a closed subspace of $C([0;T);H^s(\T^d))$.
\end{remark}

\begin{remark}
\rm

As mentioned in the introduction, the adapted solution space $\X^{s}_T$ for $u$ solving \eqref{EQ} is the replacement of the solution space $D^W(H^s)$ used by Chouk and Gubinelli in the context of modulated Schr\"odinger and KdV equations; see \cite[Definition 2.2]{CG1} and the more detailed presentation in \cite{CGLLO}.
\end{remark}

\begin{proposition}\label{PROP:UpVp}
Let $s\in\R$ and $r\in[1;\infty]$. Then the following hold:\\
\textup{(i)} We have the continuous embeddings
\begin{align*}
U^2_{W\L}(I)H^s(\T^d) \subset \X^{s}(I)\subset \Y^{s}(I)\subset V^2_{W\L}(I)H^s(\T^d).
\end{align*}

\noi
\textup{(ii)} For any $T>0$ and $\psi\in H^s(\T^d)$, we have $e^{-(W_t-W_0)\L}\psi\in \X^s_T$ and
\begin{align*}
\|e^{-(W_t-W_0)\L}\psi\|_{\X^s_T}\les \|\psi\|_{H^s}.
\end{align*}

\noi
\textup{(iii)} For any $T>0$ and $F\in L^1([0;T);H^s(\T^d))$, we have $t\mapsto \ind_{[0;\infty)}(t)\int_0^te^{-(W_t-W_s)\L}F(s)ds\in \X^s_T$ and
\begin{align*}
\Big\|\int_0^te^{-(W_t-W_s)\L}F(s)ds\Big\|_{\X^{s}_T}\le \sup_{\substack{w\in \Y^{-s}_T\\\|w\|_{\Y^{-s}_T}\le 1}}\Big|\int_0^T\int_{\T^d} F(t,x)\cj w(t,x)dxdt\Big|.
\end{align*}
\end{proposition}

\subsubsection{\textbf{Short-time function spaces}}
We now introduce a short-time version of the $U^2$ type spaces of the previous subsubsection. We refer to \cite{SchippaThesis} for the properties of these spaces.
\begin{definition}\label{DEF:F}
\rm~~\\
\textup{(i)} Given $s\in\R$, $\nu>0$, and a time interval $I\subset \R$, the short-time $U^2$ space $\Fb^{s,\nu}(I)$ is the Banach space defined through the norm
\begin{align*}
\|u\|_{\Fb^{s,\nu}(I)}^2:=\sum_N N^{2s}\sup_{\substack{I_N\subset I\\ |I_N|=\min(N^{-\nu};|I|)}}\big\|\P_N u\big\|_{U^2_{W\L}(I_N)L^2_x}^2,
\end{align*}
where the supremum is taken over all subintervals $I_N$ of $I$ of length $|I_N|=\min(N^{-\nu};|I|)$.
~~\\
\textup{(ii)} Given $s\in\R$, $\nu>0$, and a time interval $I\subset \R$, the short-time $DU^2$ space $\Nb^{s,\nu}(I)$ is the Banach space defined through the norm
\begin{align*}
\|u\|_{\Nb^{s,\nu}(I)}^2:=\sum_N N^{2s}\sup_{\substack{I_N\subset I\\ |I_N|=\min(N^{-\nu};|I|)}}\big\|\P_N u\big\|_{DU^2_{W\L}(I_N)L^2_x}^2.
\end{align*}
~~\\
\textup{(iii)} Given $s\in\R$, and a time interval $I\subset \R$, the energy space $\Eb^s(I)$ is the Banach space defined through the norm
\begin{align*}
\|u\|_{\Eb^{s}(I)}^2:=\sum_N N^{2s}\big\|\P_N u\big\|_{L^\infty_IL^2_x}^2.
\end{align*}
\end{definition}

\begin{remark}
\rm From Proposition~\ref{PROP:UpVp0}~(ii) we see that $\|u\|_{\Fb^{s,\nu}_T}\le\|u\|_{\X^s_T}$ for any $s\in\R$ and $\nu\ge 0$.
\end{remark}

These spaces enjoy the following linear estimates.
\begin{lemma}\label{LEM:FNE}
Let $u$ be a solution to the linear equation 
\begin{align*}
\dt u +\frac{dW_t}{dt}\L u=f.
\end{align*}  
Then for $s\in\R$, $\nu>0$, and a time interval $I\subset \R$, it holds
\begin{align}\label{linear}
\|u\|_{\Fb^{s,\nu}(I)}\les \|u\|_{\Eb^s(I)}+\|f\|_{\Nb^{s,\nu}(I)}.
\end{align}
Moreover,
\begin{align}\label{embedding}
\|u\|_{L^\infty_I H^s}\les \|u\|_{\Eb^s_T}\les \|u\|_{\Fb^{s,\nu}(I)}.
\end{align}
\end{lemma}
\begin{proof}
The estimate \eqref{linear} simply follows from writing the Duhamel formula on each time interval $I_{N_0}=[a;b)\subset[0;T)$ of length $|I_{N_0}|=N_0^{-\nu}\wedge T$ for a dyadic integer $N_0$, and using Proposition~\ref{PROP:UpVp}~(ii) to estimate the linear term, and Definition~\ref{DEF:DU} for the $U^2$ norm of $\int_a^te^{W_\tau\L}f(\tau)d\tau$.

The first part of estimate \eqref{embedding} follows from Minkowski's inequality, while the second one from Proposition~\ref{PROP:UpVp0}~(iii) on each time interval $I_{N_0}\subset[0;T)$, $N_0$ dyadic integer. We refer to \cite{SchippaThesis} for the details.
\end{proof}

\section{The strongly non-resonant case}\label{SEC:NR}
In this section, we give the proof of Theorem~\ref{THM:LWP-NR}. It relies on the following multilinear estimate.
\begin{lemma}\label{LEM:Multi}
Assume that {\ref{A0}} holds for some $\gamma>\frac12$ and $\rho>0$, and let $\NN$ satisfy {\ref{A3}-\ref{A4}$_{\beta}$-\ref{A5}$_\alpha$}. Then for any interval $I\subset \R$, any $N_0,\dots,N_m$ dyadic integers, and any $u_j\in U^2_{W\L}(I)L^2$, $j=1,...,m$, and $u_0\in V^2_{W\L}(I)L^2$, it holds
\begin{multline}\label{multi0}
\Big|\int_I\int_{\T^d}\NN(\P_{N_1}u_1,\dots,\P_{N_m}u_m)\cj{\P_{N_0}u_0} dxdt\Big|\\\les |I|^\gamma (N_0^*)^{\beta_{j_0}+\beta_{j_1}-\alpha_1\rho}(N_2^*)^{\beta_{j_2}+\frac{d}2-\alpha_2\rho}\Big(\prod_{j=3}^m(N_j^*)^{\beta_j+\frac{d}2}\Big)\\\times\|\P_{N_0}u_0\|_{V^2_{W\L}(I)L^2_x}\prod_{j=1}^m\|\P_{N_j}u_j\|_{U^2_{W\L}(I)L^2_x}.
\end{multline}
\end{lemma}

\begin{proof}
From assumption \ref{A3}, we can first rewrite
\begin{align*}
\int_I\int_{\T^d}\NN(\P_{N_1}u_1,\dots,\P_{N_m}u_m)\cj{u_0} dxdt = \int_I\sum_{\Gamma^{m+1}(\{\pm_j\})}\widehat{\NN}(k_0,\dots,k_m)\prod_{j=0}^{n}\widehat{\P_{N_j}u_j}^{\pm_j}  dt.
\end{align*}

From the multilinearity of the expression above, the embedding $V^2_{W\L}(I)L^2(\T^d)\subset U^b_{W\L}(I)L^2(\T^d)$ with $2<b$ such that $b'<\frac1\gamma$, and the atomic structure of $U^2$ and $U^b$, it suffices to consider the case where for $j=1,...,m$, $u_j$ are $U^2_{W\L}(I)L^2(\T^d)$-atoms, and $u_0$ is a $U^b_{W\L}(I)L^2(\T^d)$-atom:
\begin{align*}
u_j(t,x)=\sum_{k=1}^{K^{(j)}}\mathbf{1}_{[t_{k-1}^{(j)};t_k^{(j)})}e^{W_t\L}\psi_k^{(j)},\qquad \{t_k^{(j)}\}_k\in\ZZ(I),\qquad\|\psi_k^{(j)}\|_{\ell^2_kL^2_x}\le 1
\end{align*}
for any $j=1,\dots,m$, and
\begin{align*}
u_0(t,x)=\sum_{k=1}^{K^{(0)}}\mathbf{1}_{[t_{k-1}^{(0)};t_k^{(0)})}e^{W_t\L}\psi_k^{(0)}, \qquad \{t_k^{(0)}\}_k\in\ZZ(I),\qquad \|\psi_k^{(0)}\|_{\ell^b_k L^2_x}\le 1.
\end{align*}

Then, for given $k_0,\dots,k_{n}$, letting 
\begin{align}\label{Ik}
I_{k_0,\dots,k_{n}}=\bigcap_{j=0}^{m}[t_{k_j-1}^{(j)};t_{k_j}^{(j)}),
\end{align}
using Plancherel theorem and the occupation time formula \eqref{occupation}, we estimate
\begin{align*}
&\Big|\int_{I_{k_0,\dots,k_{m}}}\sum_{\Gamma^{m+1}(\{\pm_j\})}\widehat{\NN}(k_0,\dots,k_m)\prod_{j=0}^{n}\widehat{\P_{N_j}e^{W_t\L}\psi_{k_{j}}^{(j)}}^{\pm_{j}} dt\Big|\\
&\qquad\qquad=\Big|\sum_{\Gamma^{m+1}(\{\pm_j\})}\int_{I_{k_0,\dots,k_{m}}}e^{iW_t\sum_{j=0}^m\pm_j\varphi(k_j)}dt\\
&\qquad\qquad\qquad\qquad\times\widehat{\NN}(k_0,\dots,k_m)\prod_{j=0}^{n}\widehat{\P_{N_j}\psi_{k_{j}}^{(j)}}^{\pm_{j}} \Big|\\
&\qquad\qquad= \Big|\sum_{\Gamma^{m+1}(\{\pm_j\})}\widehat{\frac{d\mu_{I_{k_0,\dots,k_m}}}{dz}}\Big(\sum_{j=0}^m\pm_j\varphi(k_j)\Big)\\
&\qquad\qquad\qquad\qquad\times\widehat{\NN}(k_0,\dots,k_m)\prod_{j=0}^{n}\widehat{\P_{N_j}\psi_{k_{j}}^{(j)}}^{\pm_{j}} \Big|\\
&\qquad\qquad\les |I_{k_0,\dots,k_{m}}|^\gamma \sum_{\Gamma^{m+1}(\{\pm_j\})}\jb{\sum_{j=0}^m\pm_j\varphi(k_j)}^{-\rho}\\
&\qquad\qquad\qquad\qquad\times|\widehat{\NN}(k_0,\dots,k_m)|\prod_{j=0}^{n}\big|\widehat{\P_{N_j}\psi_{k_{j}}^{(j)}}^{\pm_{j}}\big| \\
& \qquad\qquad\les |I_{k_0,\dots,k_{m}}|^\gamma (N_0^*)^{-\alpha_1\rho}(N_2^*)^{-\alpha_2\rho}\prod_{j=0}^mN_j^{\beta_j}\\
&\qquad\qquad\qquad\qquad\times\sum_{\Gamma^{m+1}(\{\pm_j\})}\prod_{j=0}^{m}\big|\widehat{\P_{N_j}\psi_{k_{j}}^{(j)}}^{\pm_{j}}\big| ,
\end{align*}
where the last two steps follow from assumptions \ref{A0}-\ref{A4}$_{\beta}$-\ref{A5}$_{\alpha}$. Using H\"older and Bernstein inequalities, we thus find the bound
\begin{align}\label{multi1}
&\sum_{k_0,\dots,k_m}|I_{k_0,\dots,k_{m}}|^\gamma (N_0^*)^{-\alpha_1\rho}(N_2^*)^{-\alpha_2\rho}\Big(\prod_{j=0}^mN_j^{\beta_j}\Big)\notag\\
&\qquad\qquad\times\big\|\widehat{\P_{N_0^*}\psi_{k_{j_0}}^{(j_0)}}\big\|_{L^2_x}\big\|\widehat{\P_{N_1^*}\psi_{k_{j_1}}^{(j_1)}}\big\|_{L^2_x}\prod_{\ell=2}^{m}(N_\ell^*)^{\frac{d}2}\big\|\widehat{\P_{N_\ell^*}\psi_{k_{j_\ell}}^{(j_\ell)}}\big\|_{L^2_x},
\end{align}
where for $\ell=0,\dots,m$, 
\begin{align}\label{jl}
j_\ell\in\{0,\dots,m\} \text{ is such that }N_{j_\ell}=N_\ell^*.
\end{align}

Note that $N_0^*\sim N_1^*$ for the left-hand side of \eqref{multi0} to be non-zero. Using H\"older inequality together with our choice of $b$ to sum on $k_j$'s, this yields the following bound on \eqref{multi1}:
\begin{align*}
&(N_0^*)^{\beta_{j_0}+\beta_{j_1}-\alpha_1\rho}(N_2^*)^{\beta_{j_2}+\frac{d}2-\alpha_2\rho}\Big(\prod_{j=3}^m(N_j^*)^{\beta_j+\frac{d}2}\Big)\Big\||I_{k_0,\dots,k_{m}}|^\gamma\Big\|_{\ell^{b'}_{k_0}\ell^2_{k_1,\dots,k_m}}\\
&\qquad\qquad\qquad\times\|\P_{N_0}\psi_{k_{0}}^{(0)}\|_{\ell^b_{k_0}L^2_x}\prod_{j=1}^{m}\|\P_{N_j}\psi_{k_{j}}^{(j)}\|_{\ell^2_{k_{j}}L^2_x}\\
&\qquad\les (N_0^*)^{\beta_{j_0}+\beta_{j_1}-\alpha_1\rho}(N_2^*)^{\beta_{j_2}+\frac{d}2-\alpha_2\rho}\Big(\prod_{j=3}^m(N_j^*)^{\beta_j+\frac{d}2}\Big)\Big\||I_{k_0,\dots,k_{m}}|^\gamma\Big\|_{\ell^{\frac1\gamma}_{k_0,\dots,k_m}}\\
&\qquad\qquad\qquad\times\|\P_{N_0}\psi_{k_{0}}^{(0)}\|_{\ell^b_{k_0}L^2_x}\prod_{j=1}^{m}\|\P_{N_j}\psi_{k_{j}}^{(j)}\|_{\ell^2_{k_{j}}L^2_x}\\
&\qquad\sim |I|^\gamma (N_0^*)^{\beta_{j_0}+\beta_{j_1}-\alpha_1\rho}(N_2^*)^{\beta_{j_2}+\frac{d}2-\alpha_2\rho}\Big(\prod_{j=3}^m(N_j^*)^{\beta_j+\frac{d}2}\Big)\\
&\qquad\qquad\qquad\times\|\P_{N_0}\psi_{k_{0}}^{(0)}\|_{\ell^b_{k_0}L^2_x}\prod_{j=1}^{m}\|\P_{N_j}\psi_{k_{j}}^{(j)}\|_{\ell^2_{k_{j}}L^2_x},
\end{align*}
where we used that
$$\sum_{k_0,\dots,k_m}|I_{k_0,\dots,k_m}|=T$$ from the definition \eqref{Ik} of $I_{k_0,\dots,k_m}$.

This finally proves \eqref{multi0} due to the atomic structure of $U^2_{W\L}(I)L^2(\T^d)$ and $U^b_{W\L}(I)L^2(\T^d)\supset V^2_{W\L}(I)L^2(\T^d)$.
\end{proof}
Summing the previous multilinear estimate on the $N_j$'s leads to the following.
\begin{proposition}\label{PROP:multi}
Let $s\in\R$, $\NN$ satisfy {\ref{A3}-\ref{A4}$_{\beta}$-\ref{A5}$_{\alpha}$} and assume that {\ref{A0}} holds for some $\gamma>\frac12$ and $\rho$ satisfying {\ref{Arho}$_{\alpha,\beta,s}$}.\\
Then for any $\sigma \ge s$, there is $C(\sigma)>0$ non-decreasing such that for any $T>0$, any $u_j\in C([0;T);H^\sigma(\T^d))\cap \X^\sigma_T$, $j=1,...,m$, it holds
\begin{align}\label{multi00}
\Big\|\int_0^te^{-(W_t-W_{t'})\L}\NN(u_1,\dots,u_m)dt'\Big\|_{\X^\sigma_T}\le C(\sigma)T^\gamma \sum_{j=1}^m\|u_j\|_{\X^\sigma_T}\prod_{j'\neq j}\|u_j\|_{\X^s_T}.
\end{align}
\end{proposition}
\begin{proof}
Using the multilinearity of $\NN$ we can decompose dyadically
\begin{align*}
\int_0^te^{-(W_t-W_{t'})\L}\NN(u_1,\dots,u_m)dt'=\sum_{\substack{N_0,\dots,N_m\\N_0^*\sim N_1^*}}\int_0^te^{-(W_t-W_{t'})\L}\P_{N_0}\NN(\P_{N_1}u_1,\dots,\P_{N_m}u_m)dt'.
\end{align*}
Since $u_j$ are in $C([0;T);H^\sigma(\T^d)\cap \X^{\sigma}_T$, the crude bound
\begin{align*}
\Big\|\NN(\P_{N_1}u_1,\dots,\P_{N_m}u_m)\Big\|_{L^1_TH^\sigma}\les T\max(N_1,\dots,N_m)^{\sigma+\beta_0}\prod_{j=1}^mN_j^{\beta_j+\frac{d}2-\sigma}\|\P_{N_j}u\|_{C_TH^\sigma}
\end{align*}
ensures that we can use Proposition~\ref{PROP:UpVp}~(iii). Therefore, we have
\begin{align*}
&\Big\|\int_0^te^{-(W_t-W_{t'})\L}\NN(u_1,\dots,u_m)dt'\Big\|_{\X^\sigma_T}\\
&\qquad\qquad\le \sum_{\substack{N_1,\dots,N_m}}\sup_{v_{\cj{N}}\in \Y^{-\sigma}_T}\Big|\sum_{N_0}\mathbf{1}_{N_0^*\sim N_1^*}\int_0^T\int_{\T^d}\NN(\P_{N_1}u_1,\dots,\P_{N_M}u_m)\cj{\P_{N_0}v_{\cj{N}}}dxdt\Big|,
\end{align*}
where we write $\cj{N}=(N_1,\dots,N_m)$.

Then, since $\P_{N_j}u_j\in U^2_{W\L}(T)L^2$ and $P_{N_0}v_{\cj{N}}\in V^2_{W\L}(T)L^2$, invoking Lemma~\ref{LEM:Multi}, we find that this is bounded by 
\begin{multline}\label{multi01}
T^\gamma\sum_{\substack{N_0,\dots,N_m\\ N_0^*\sim N_1^*}}(N_0^*)^{\beta_{j_0}+\beta_{j_1}-\alpha_1\rho}(N_2^*)^{\beta_{j_2}+\frac{d}2-\alpha_2\rho}\Big(\prod_{j=3}^m(N_j^*)^{\beta_{j_\ell}+\frac{d}2}\Big)\\
\times\|\P_{N_0}v_{\cj{N}}\|_{V^2_{W\L}(T)L^2_x}\prod_{j=1}^m\|\P_{N_j}u_j\|_{U^2_{W\L}(T)L^2_x},
\end{multline}
where we recall that $j_\ell$ are defined by $N_{j_\ell}=N_\ell^*$. It thus remains to perform the summation on the $N_j$'s.

\textbf{Case 1: $N_0\sim N_0^*$.}\\
 Then without loss of generality we can assume that $N_0=N_0^*$, and under the assumption
\begin{align*}
\rho\ge \alpha_1^{-1}\big(\beta_0+\beta_1^*\big),
\end{align*}
we can sum on $N_0$ using Cauchy-Schwartz inequality, to estimate
\begin{align*}
&T^\gamma\sum_{\substack{N_0,\dots,N_m\\N_0= N_0^*\sim N_1^*}}(N_0^*)^{\beta_{0}+\beta_{j_1}-\alpha_1\rho}(N_2^*)^{\beta_{j_2}+\frac{d}2-\alpha_2\rho -s}\Big(\prod_{j=3}^m(N_j^*)^{\beta_j+\frac{d}2-s}\Big)\\
&\qquad\times N_0^{-\sigma}\|\P_{N_0}v_{\cj{N}}\|_{V^2_{W\L}(T)L^2_x}(N_1^*)^{\sigma}\|\P_{N_1^*}u_{j_1}\|_{U^2_{W\L}(T)L^2_x}\prod_{\ell=2}^m(N_\ell^*)^s\|\P_{N_\ell^*}u_{j_\ell}\|_{U^2_{W\L}(T)L^2_x}\\
&\les T^\gamma\|u_{j_1}\|_{\X^\sigma_T}\sum_{\substack{N_j\\j\not\in\{0,j_1\}}}(N_2^*)^{\beta_{0}+\beta_{j_1}+\beta_{j_2}+\frac{d}2-s-(\alpha_1+\alpha_2)\rho}\Big(\prod_{\ell=3}^m(N_\ell^*)^{\beta_{j_\ell}+\frac{d}2-s}\Big)\\
&\qquad\qquad\times(N_2^*)^{s}\|\P_{N_2^*}u_{j_2}\|_{U^2_{W\L}(T)L^2_x}\prod_{\ell=3}^m(N_\ell^*)^{s}\|\P_{N_\ell^*}u_{j_\ell}\|_{U^2_{W\L}(T)L^2_x},
\end{align*}
where we used that ${\displaystyle \|v_{\cj{N}}\|_{\Y^{-\sigma}_T}\le 1}$.

Now provided that
\begin{align*}
\rho>(\alpha_1+\alpha_2)^{-1}\big(\beta_0+\beta_1^*+\sum_{j= 2}^{n}(\beta_j^*+\frac{d}2-s)\big) \text{ for any }2\le n\le m,
\end{align*}
we can continue summing on $N_2^*,\dots,N_m^*$ to bound \eqref{multi01} by
\begin{align}\label{multi0RHS}
T^\gamma\sum_{j=1}^m\|u_j\|_{\X^\sigma_T}\prod_{j'\neq j}\|u_{j'}\|_{\X^s_T}.
\end{align}

\textbf{Case 2: if $N_0\sim N_{2}^*$.}\\
We now have to bound
\begin{multline}\label{multi02}
T^\gamma\sum_{\substack{N_0,\dots,N_m\\ N_0^*\sim N_1^*}}(N_0^*)^{\beta_{j_0}+\beta_{j_1}-\alpha_1\rho-s-\sigma}(N_2^*)^{-\alpha_2\rho+\beta_0+\frac{d}2+\sigma}\Big(\prod_{\substack{\ell=3}}^m(N_\ell^*)^{\beta_{j_\ell}+\frac{d}2-s}\Big)\\
\times N_0^{-\sigma}\|\P_{N_0}v_{\cj{N}}\|_{V^2_{W\L}(T)L^2_x}(N_0^*)^\sigma\|\P_{N_0^*}u_{j_0}\|_{U^2_{W\L}(T)L^2_x}\prod_{\substack{j=1\\j\neq j_0}}^mN_j^s\|\P_{N_j}u_j\|_{U^2_{W\L}(T)L^2_x}.
\end{multline}
Summing on $N_0^*$, we bound \eqref{multi02} by
\begin{multline*}
T^\gamma\sum_{\substack{N_\ell^*, \ell\ge 2}}(N_2^*)^{\beta_{j_0}+\beta_{j_1}+\beta_{0}+\frac{d}2-s-(\alpha_1+\alpha_2)\rho}\Big(\prod_{\substack{\ell=3}}^m(N_\ell^*)^{\beta_{j_\ell}+\frac{d}2-s}\Big)\\
\times \|u_{j_0}\|_{\X^\sigma_T}\|u_{j_1}\|_{\X^s_T}N_0^{-\sigma}\|\P_{N_0}v_{\cj{N}}\|_{V^2_{W\L}(T)L^2_x}\prod_{\substack{\ell=3}}^m(N_\ell^*)^s\|\P_{N_\ell^*}u_{j_\ell}\|_{U^2_{W\L}(T)L^2_x}.
\end{multline*}
provided that
\begin{align*}
\rho \ge \alpha_1^{-1}(\beta_1^*+\beta_2^*-2s).
\end{align*}
If moreover
\begin{align*}
\rho>\big(\alpha_1+\alpha_2)^{-1}\big(\beta_1^*+\beta_2^*+\beta_0+\frac{d}2-s+\sum_{j=3}^n(\beta_j^*+\frac{d}2-s)\big)\text{ for any }3\le n\le m,
\end{align*}
we can keep summing on $N_2^*,\dots,N_m^*$ to bound \eqref{multi02} by \eqref{multi0RHS}.\\
%\noindent\textbf{Subcase (b): if $\sigma<0$.} In this case, we have $N_0^{\sigma}\les (N_m^*)^\sigma$, and so provided that
%\begin{align*}
%\rho>\alpha_1^{-1}\big(\beta_0^*+\beta_1^*-2s\big)\qquad\text{and}\qquad \rho>(\alpha_1+\alpha_2)^{-1}\big(\beta_0^*+\beta_1^*+\sum_{j\ge 2}\max\big(\beta_j^*+\frac{d}2-s;0\big)\big),
%\end{align*}
%we can proceed as above to estimate
%\begin{align*}
%&T^\gamma\sum_{\substack{N_0,\dots,N_m\\N_0^*\sim N_1^*}}(N_0^*)^{\beta_{j_0}+\beta_{j_1}-\alpha_1\rho}(N_2^*)^{\beta_{j_2}+\frac{d}2-\alpha_2\rho}\Big(\prod_{j=3}^m(N_j^*)^{\beta_j+\frac{d}2}\Big)\\
%&\qquad\qquad\qquad\times\|\P_{N_0}u_0\|_{V^2_{W\L}(T)L^2_x}\prod_{j=1}^m\|\P_{N_j}u_j\|_{U^2_{W\L}(T)L^2_x}\\
%&\qquad\les T^\gamma\|u_{j_0}\|_{X^{\sigma}_T}\|u_{j_1}\|_{\X^s_T}\sum_{N_2^*\ge \dots\ge N_m^*}(N_2^*)^{\beta_{j_0}+\beta_{j_1}-(\alpha_1+\alpha_2)\rho+\beta_{j_2}+\frac{d}2-s-\sigma }\|\P_{N_2^*}u_{j_2}\|_{U^2_{W\L}(T)L^2_x}\\
%&\qquad\qquad\times N_0^{\beta_0+\frac{d}2}\|\P_{N_0}u_0\|_{V^2_{W\L}(T)L^2_x}\prod_{\substack{\ell\ge 2\\\ell\neq\ell_0}}(N_{\ell}^*)^{\beta_{j_\ell}+\frac{d}2}\|\P_{N_\ell^*}u_{j_\ell}\|_{U^2_{W\L}(T)L^2_x}\\
%&\qquad\les T^\gamma\sum_{j=1}^m\|u_j\|_{\X^\sigma_T}\prod_{j'\neq j,0}\|u_{j'}\|_{\X^s_T}\\
%&\qquad\qquad\qquad\times\sum_{N_0}(N_0)^{\beta_{j_0}+\beta_{j_1}-(\alpha_1+\alpha_2)\rho+\sum_{\ell\ge 2}\max\big(\beta_{j_\ell}+\frac{d}2-s;0\big)-\sigma}\|\P_{N_0}u_0\|_{V^2_{W\L}(T)L^2_x}\\
%&\qquad\les T^\gamma\|u_0\|_{\Y^{-\sigma}_T}\sum_{j=1}^m\|u_j\|_{\X^\sigma_T}\prod_{j'\neq j}\|u_{j'}\|_{\X^s_T}.
%\end{align*}
\textbf{Case 3: $N_0\les N_3^*$.} The other cases $N_0\sim N_3^*,\dots, N_0\sim N_m^*$ are treated similarly as above, provided that for any $\ell_0\ge 2$ such that $N_0\sim N_{\ell_0}^*$,
\begin{align*}
\begin{cases}
\rho\ge \alpha_1^{-1}\big(\beta_1^*+\beta_2^*-2s),\\
{\displaystyle \rho>(\alpha_1+\alpha_2)^{-1}\big(\beta_1^*+\beta_2^*-2s+\sum_{\ell=2}^{L}(\beta_{\ell+1}^*+\frac{d}2-s)\big) \text{ for any }2\le L\le \ell_0-1,}\\
{\displaystyle \rho>(\alpha_1+\alpha_2)^{-1}\big(\beta_1^*+\beta_2^*+\beta_0+\frac{d}2-s+\sum_{\ell=2}^{L}(\beta_{\ell+1}^*+\frac{d}2-s)\big) \text{ for any }\ell_0+1\le L\le m.}\\
\end{cases}
\end{align*}
This finally shows \eqref{multi00}.
\end{proof}
With Proposition~\ref{PROP:multi} at hand, we can now finish the proof of Theorem~\ref{THM:LWP-NR}.
\begin{proof}[Proof of Theorem~\ref{THM:LWP-NR}]
For $u_0\in H^{s}(\T^d)$, let $R=2\|u_0\|_{H^{s}}$, and $T\in(0;1]$ to be chosen later, and define 
$$B_R(T):=\big\{u\in C([0;T);H^{s}(\T^d))\cap \X^s_T,~~\|u\|_{\X^s_T}\le R\big\}$$
which is closed in $\X^s_T$, and
\begin{align*}
\Psi_{u_0}: u\in \X^s_T\mapsto e^{(W_t-W_0)\L}u_0+ \int_0^t e^{(W_t-W_{\tau})\L}\NN(u)(\tau)d\tau.
\end{align*}
We will show that $\Psi_{u_0}$ defines a contraction on $B_R(T)$. From Proposition~\ref{PROP:UpVp}, the multilinearity of $\NN$, and Proposition~\ref{PROP:multi}, we have for any $u,v\in B_R(T)$:
\begin{align}\label{FixPoint}
\big\|\Psi_{u_0}(u)-\Psi_{v_0}(v)\big\|_{\X^s_T}&\les \|u_0-v_0\|_{H^{s}}+\Big\|\int_0^te^{-(W_t-W_{t'})\L}\big(\NN(u)-\NN(v)\big)dt'\Big\|_{\X^s_T}\notag\\
&\le C\|u_0-v_0\|_{H^{s}} +CT^\gamma \|u-v\|_{\X^s_T}\big(\|u\|_{\X^s_T}+\|v\|_{\X^s_T}\big)^{m-1}.
\end{align}
Thus, taking $T(R)\sim \jb{R}^{-\frac{m-1}\gamma}$ ensures that $\Psi_{u_0}$ defines a contraction on $B_R(T)$, thus providing a mild solution $u\in C([0;T);H^{s}(\T^d))\cap\X^s_T$ to \eqref{EQ} on $[0;T)$, unique in $B_R(T)$. If $v\in C([0;T);H^s(\T^d))\cap \X^s_T$ is another solution, with $R'=\max\big(R;\|v\|_{\X^s_T}\big)$ and $\tau=\tau(R')\le T$, the argument above shows that $v\equiv u$ on $[0;\tau]$. But then iterating this argument on $[\tau;2\tau]$, etc., shows that $v\equiv u$ on the whole time interval where $u$ is defined, showing uniqueness in the whole class $C([0;T);H^s(\T^d))\cap \X^s_T$. The estimate \eqref{FixPoint} also provides the Lipschitz dependence of the flow map on the initial data. At last, using Proposition~\ref{PROP:multi} again with $\sigma\ge s$ yields
\begin{align*}
\|u\|_{\X^\sigma_T}=\|\Psi_{u_0}(u)\|_{\X^\sigma_T} \le C\|u_0\|_{H^\sigma}+CT^\gamma\|u\|_{\X^\sigma_T}\|u\|_{\X^s_T}^{m-1},
\end{align*} 
leading to
\begin{align*}
\|u\|_{\X^\sigma_T}\le 2C\|u_0\|_{H^\sigma}
\end{align*}
with the same choice of $T=T(\|u_0\|_{H^s})$ as above. This shows the persistence of regularity property. 

%At last, if \ref{A2} holds, Lemma~\ref{LEM:L2} below ensures that the $L^2$ norm is also invariant under the flow of \eqref{EQ} we just built, allowing to iterate the local well-posedness in $L^2(\T^d)$
\end{proof}

\section{Perturbation by completely resonant terms}
In this section, we proceed with the proof of Theorem~\ref{THM:LWP-R}. As explained in the introduction, we mainly follow the argument in \cite{OW}.

\subsection{Existence}
We start by proving existence of a mild solution to \eqref{EQ}, under the assumptions of Theorem~\ref{THM:LWP-R}. Thus, fix $s\in\R$ and $\rho>0$ satisfying \ref{Arho*}$_{\alpha,\beta,\beta_\RR,s}$, with $\nu>0$ as in \eqref{nu}. 

We first have the following well-posedness result at higher regularity.
\begin{lemma}\label{LEM:LWP}
Given $s\in\R$, $\NN=\NN_0+\RR$ with $\NN_0$ satisfying {\ref{A3}}-{\ref{A4}$_{\beta}$}-{\ref{A5}$_{\alpha}$}, and $\RR$ satisfying {\ref{A3}}-{\ref{A4}$_{\beta_\RR}$}-\eqref{Resonant}, provided that $\rho$ satisfies {\ref{Arho}$_{\alpha,\beta,s_\RR}$} with $s_\RR=\max(s;\frac{|\beta_\RR|}{m-1})$, then \eqref{EQ} is well-posed in $H^\sigma(\T^d)$ for any $\sigma\ge s_\RR$. More precisely, for any $u_0\in H^\sigma(\T^d)$, there is $T=T(\|u_0\|_{H^{s_\RR}})\in(0;1]$ and a unique mild solution $u\in C([0;T);H^\sigma(\T^d))\cap \X^\sigma_T$ to \eqref{EQ}. Moreover the flow map is Lipschitz continuous.
\end{lemma}
\begin{proof}
From the assumption on $\NN_0$ and $\RR$, we have by Proposition~\ref{PROP:multi} that for any $\sigma\ge s_\RR$,
\begin{align}\label{multi2}
\Big|\int_0^T\int_{\T^d}\NN_0(u_1,\dots,u_m)\cj{w} dxdt\Big|\les T^\gamma\|w\|_{\Y^{-\sigma}_T}\sum_{j=1}^m\|u_j\|_{\X^\sigma_T}\prod_{j'\neq j}^{n}\|u_{j'}\|_{\X^s_T},
\end{align}
provided that $\rho$ satisfies \ref{Arho}$_{\alpha,\beta,s_\RR}$.

Moreover, from the assumptions on $\RR$ in \ref{A5*}, we have for any $\sigma\ge \frac{|\beta_\RR|}{m-1}$,
\begin{align}\label{multi3}
\Big|\int_0^T\int_{\T^d}\RR(u_1,\dots,u_m)\cj{w} dxdt\Big|&=\Big|\int_0^T\sum_{k_0\in\Z^d}\ft\RR(k_0,\dots,k_0)\big(\prod_{j=1}^m\ft u_j(k_0)^{\pm_j}\big)\cj{\ft w(k_0)}dt\Big|\notag\\
&\les T\Big\|\sum_{k_0\in\Z^d}\jb{k_0}^{|\beta_\RR|}|\ft w(k_0)|\prod_{j=1}^m|\ft u_j(k_0)|\Big\|_{L^\infty_T}\notag\\
&\les T\|w\|_{L^\infty_TH^{-\sigma}}\sum_{j=1}^m\|u_j\|_{L^\infty_T H^\sigma}\prod_{j'\neq j}^m\|u_{j'}\|_{L^\infty_T H^{\frac{|\beta_\RR|}{m-1}}}\notag\\
&\les T\|w\|_{\Y^{-\sigma}_T}\sum_{j=1}^m\|u_j\|_{\X^\sigma_T}\prod_{j'\neq j}\|u_{j'}\|_{\X^{s_\RR}_T}.
\end{align}
Then, with \eqref{multi2}-\eqref{multi3} at hand, the same argument as in the proof of Theorem~\ref{THM:LWP-NR}~(i) in Section~\ref{SEC:NR} above shows well-posedness of \eqref{EQ} in $H^\sigma(\T^d)$ in this case for any $\sigma\ge s_\RR$, with a time of existence $T=T(\|u_0\|_{H^{s_\RR}})$, and Lipschitz continuous dependence on the initial data.
\end{proof}

Lemma~\ref{LEM:LWP} thus establishes Theorem~\ref{THM:LWP-R}~(i). To control the resonant part of the nonlinearity at lower regularity, we make use of the frequency-dependent short-time version $\Fb^{s,\nu}_T$ of $\X^s_T$ in Definition~\ref{DEF:F}. We will then take 
$\nu$ as in \eqref{nu} and $\rho$ depending on $\nu$ as in \ref{Arho*}$_{\alpha,\beta,\beta_\RR,s}$.
\begin{lemma}\label{LEM:N}
Under the assumptions of Lemma~\ref{LEM:LWP}, if moreover $\nu$ is given by \eqref{nu} and $\rho$ satisfies assumption \ref{Arho*}$_{\alpha,\beta,\beta_\RR,s}$, then there exists $0<\theta\ll 1$ such that for any $T>0$, $\sigma\ge s$, and $u_j\in \Fb^{s,\nu}_T$, $j=1,\dots,m$, it holds
\begin{align}\label{N}
\big\|\NN(u_1,\dots,u_m)\big\|_{\Nb^{\sigma,\nu}_T}\les T^\theta \sum_{j=1}^m\|u_j\|_{\Fb^{\sigma,\nu}_T}\prod_{j'\neq j}\|u_{j'}\|_{\Fb^{s,\nu}_T}.
\end{align}
\end{lemma}
\begin{proof}
From Proposition~\ref{PROP:duality}, the multilinearity of $\NN_0$, and the support property \eqref{Resonant} of $\RR$, we have
\begin{align*}
&\big\|\NN(u_1,\dots,u_m)\big\|_{\Nb^\sigma_T}\\
&=\Big\|N_0^\sigma \sup_{\substack{I_{N_0}\subset [0;T)\\|I_{N_0}|\sim N_0^{-\nu}}}\big\|\P_{N_0}\NN(u_1,\dots,u_m)\big\|_{DU^2_{W\L}(I_{N_0})L^2_x}\Big\|_{\ell^2_{N_0}}\\
&\qquad= \Big\|N_0^\sigma\sup_{\|w\|_{V^2_{W\L}(I_{N_0})L^2_x}\le 1}\big|\int_{I_{N_0}}\int_{\T^d}\Big(\NN_0(u_1,\dots,u_m)+\RR(u_1,\dots,u_m)\Big)\cj{\P_{N_0} w}dxdt\big|\Big\|_{\ell^2_{N_0}}\\
&\qquad\le \Big\|N_0^\sigma\sup_{\|w\|_{V^2_{W\L}(I_{N_0})L^2_x}\le 1}\Big\{\sum_{N_1,\dots,N_m}\Big|\int_{I_{N_0}}\int_{\T^d}\NN_0(\P_{N_1}u_1,\dots,\P_{N_m}u_m)\cj{\P_{N_0} w}dxdt\Big|\\
&\qquad\qquad\qquad\qquad+\Big|\int_0^t\int_{\T^d}\RR(\P_{N_0}u_1,\dots,\P_{N_0}u_m)\cj{\P_{N_0} w}dxdt\Big|\Big\}\Big\|_{\ell^2_{N_0}}.
\end{align*}
Cutting the time interval $I_{N_0}$ in ${\displaystyle O\Big(\big(\frac{N_0^*}{N_0}\big)^\nu\Big)}$ intervals of size $|I_{N_0^*}|\le (N_0^*)^{-\nu}\wedge |I_{N_0}|$, then invoking Lemma~\ref{LEM:Multi}, we have for any $\theta>0$:
\begin{align*}
&\sum_{N_1,\dots,N_m}\Big|\int_{I_{N_0}}\int_{\T^d}\NN_0(\P_{N_1}u_1,\dots,\P_{N_m}u_m)\cj{\P_{N_0} w}dxdt\Big|\\
&\qquad\les\sum_{N_1,\dots,N_m}\sup_{\substack{I_{N_0^*}\subset I_{N_0}\\|I_{N_0}^*|=(N_0^*)^{-\nu}\wedge |I_{N_0}|}}(N_0^*)^\nu N_0^{-\nu}\Big|\int_{I_{N_0^*}}\int_{\T^d}\NN_0(\P_{N_1}u_1,\dots,\P_{N_m}u_m)\cj{\P_{N_0} w}dxdt\Big|\\
&\qquad\les T^\theta\sum_{\substack{N_1,\dots,N_m\\ N_0^*\sim N_1^*}}\sup_{\substack{I_{N_0^*}\subset I_{N_0}\\|I_{N_0}^*|=(N_0^*)^{-\nu}\wedge |I_{N_0}|}}(N_0^*)^{(1-\gamma+\theta)\nu+\beta_{j_0}+\beta_{j_1}-\alpha_1\rho}(N_2^*)^{\beta_{j_2}+\frac{d}2-\alpha_2\rho}\\
&\qquad\qquad\times\Big(\prod_{\ell=3}^m(N_\ell^*)^{\beta_{j_\ell}+\frac{d}2}\Big)N_0^{-\nu}\|\P_{N_0}w\|_{V^2_{W\L}(I_{N_0^*})L^2_x}\prod_{j=1}^m\|\P_{N_j}u_j\|_{U^2_{W\L}(I_{N_0^*})L^2_x}\\
&\qquad\les T^\theta\sum_{\substack{N_1,\dots,N_m\\ N_0^*\sim N_1^*}}(N_0^*)^{(1-\gamma+\theta)\nu+\beta_{j_0}+\beta_{j_1}-\alpha_1\rho}(N_2^*)^{\beta_{j_2}+\frac{d}2-\alpha_2\rho}N_0^{-\nu}\Big(\prod_{\ell=3}^m(N_\ell^*)^{\beta_{j_\ell}+\frac{d}2}\Big)\\
&\qquad\qquad\times\|\P_{N_0}w\|_{V^2_{W\L}(I_{N_0^*})L^2_x}\prod_{j=1}^m\Big(\sup_{\substack{I_{N_j}\subset [0;T)\\|I_{N_j}|=(N_j)^{-\nu}\wedge T}}\|\P_{N_j}u_j\|_{U^2_{W\L}(I_{N_j})L^2_x}\Big),
\end{align*}
%\begin{align*}
%&\qquad\les T^\gamma\sum_{\substack{N_1,\dots,N_m\\ N_0^*\sim N_1^*}} (N_0^*)^{\beta_{j_0}+\beta_{j_1}-\alpha_1\rho+(1-\gamma+\theta)\nu}(N_2^*)^{\beta_{j_2}+\frac{d}2-s-\alpha_2\rho}\\
%&\qquad\qquad\qquad\qquad\times N_0^{\sigma-\nu}\Big(\prod_{\substack{\ell=3\\j_\ell\neq 0}}^m(N_\ell^*)^{\beta_{j_\ell}+\frac{d}2-s}\Big) N_0^{-\sigma}\|\P_{N_0}w\|_{V^2_{W\L}(I_{N_0})L^2_x}\\
%&\qquad\qquad\qquad\qquad\times\Big((N_1^*)^\sigma\sup_{\substack{I_{N_1^*}\subset [0;T)\\|I_{N_1^*}|\le (N_1^*)^{-\nu}\wedge T}}\|\P_{N_1^*}u_{j_1}\|_{U^2_{W\L}(I_{N_1^*})L^2_x}\Big)\\
%&\qquad\qquad\qquad\qquad\times\Big(\prod_{\ell=2}^m (N_\ell^*)^{s}\sup_{\substack{I_{N_\ell^*}\subset [0;T)\\|I_{N_\ell^*}|\le (N_\ell^*)^{-\nu}\wedge T}}\|\P_{N_\ell^*}u_{j_\ell}\|_{U^2_{W\L}(I_{N_{j_\ell}})L^2_x}\Big),
%\end{align*}

where in the last step we used Proposition~\ref{PROP:UpVp0}~(ii) to localize on time intervals at the correct scale.

We can then proceed as in the proof of Proposition~\ref{PROP:multi} to sum on the $N_j$'s: this yields
\begin{align*}
\big\|\NN_0(u)\big\|_{\Nb^{\sigma,\nu}_T}\les T^\theta\sum_{j=1}^m\|u_j\|_{\Fb^{\sigma,\nu}_T}\prod_{j'\neq j}\|u_{j'}\|_{\Fb^{s,\nu}_T}
\end{align*}
for any $\sigma\ge s$, provided that
\begin{align*}
\begin{cases}
\rho> \alpha_1^{-1}\big(\beta_0+\beta_1^*-\gamma\nu),\\
\rho> \alpha_1^{-1}\big(\beta_1^*+\beta_2^*-2s+(1-\gamma)\nu),\\
{\displaystyle \rho>(\alpha_1+\alpha_2)^{-1}\big(\beta_1^*+\beta_2^*-2s+\sum_{\ell=2}^{L}(\beta_{\ell+1}^*+\frac{d}2-s)\big)+(1-\gamma)\nu,}\\
\qquad\qquad\qquad\qquad \text{ for any } 2\le \ell_0\le m\text{ and }2\le L\le \ell_0-1,\\
{\displaystyle \rho>(\alpha_1+\alpha_2)^{-1}\big(\beta_1^*+\beta_2^*+\beta_0+\frac{d}2-s+\sum_{\ell=2}^{L}(\beta_{\ell+1}^*+\frac{d}2-s)\big)-\gamma\nu},\\
\qquad\qquad\qquad\qquad\text{ for any }2\le \ell_0\le m\text{ and }\ell_0+1\le L\le m,
\end{cases}
\end{align*}
and $0<\theta\ll 1$ so that these conditions also hold with $\gamma-\theta$ in place of $\gamma$.

We can now estimate the resonant part as
\begin{align*}
&\big\|\RR(u_1,\dots,u_m)\big\|_{\Nb^{s,\nu}_T}^2\\
&\le\sum_{N}N^{2s}\sup_{\substack{I_N\subset [0;T)\\|I_N|=N^{-\nu}\wedge T}}\sup_{\|w\|_{V^2_{W\L}(I_N)L^2_x}\le 1}\Big|\int_{I_N}\sum_{k_0\in\Z^d}\ft\RR(k_0,\dots,k_0)\cj{\widehat{\P_N w}_{k_0}}\prod_{j=1}^m(\ft u_j)_{k_0}^{\pm_j}dt\Big|^2,
\end{align*}
and
\begin{align*}
&\Big|\int_{I_N}\sum_{k_0\in\Z^d}\ft\RR(k_0,\dots,k_0)\cj{\widehat{\P_N w}_{k_0}}\prod_{j=1}^m(\ft u_j)_{k_0}^{\pm_j}dt\Big|\\
&\qquad\les T^\theta N^{|\beta_\RR|-(1-\theta)\nu}\Big\|\sum_{|k_0|\sim N}|\cj{\widehat{\P_N w}_{k_0}}|\prod_{j=1}^m|(\ft u_j)_{k_0}|\Big\|_{L^\infty_{I_N}}\\
&\qquad\les T^\theta N^{|\beta_\RR|-(1-\theta)\nu}\sum_{\substack{N_1,\dots,N_m\\N_j\sim N}}\|\P_N w\|_{L^\infty_{I_N}L^2_x}\prod_{j=1}^m\|\P_{N_j}u_j\|_{L^\infty_{I_{N}}L^2_x}\\
&\qquad\les T^\theta N^{|\beta_\RR|-(1-\theta)\nu}\sum_{\substack{N_1,\dots,N_m\\N_j\sim N}}\|\P_N w\|_{V^2_{W\L}(I_N)L^2_x}\prod_{j=1}^m\|\P_{N_j}u_j\|_{U^2_{W\L}(I_{N})L^2_x}
\end{align*}
for any $\theta\in[0;1]$.

Thus after summing on $N$ we find
\begin{align*}
\big\|\RR(u_1,\dots,u_m)\big\|_{\Nb^{s,\nu}_T}&\les T^\theta \sum_{j=1}^m\|u_j\|_{\Fb^{\sigma,\nu}_T}\prod_{j'\neq j}\|u_{j'}\|_{\Fb^{s,\nu}_T},
\end{align*}
provided that $\nu$ satisfies \eqref{nu} and then $\theta$ is chosen smaller if necessary so that $$ 0<\theta<1-\frac{|\beta_\RR|-(m-1)s}{\nu}.$$
\end{proof}

We now establish an a priori energy estimate on smooth solutions to \eqref{EQ}.
\begin{lemma}\label{LEM:energy}
Under the assumptions of Theorem~\ref{THM:LWP-R}, then for $u_0\in H^\infty(\T^d)$, the smooth solution $u\in \bigcap_{\sigma \ge s_\RR}C([0;T);H^\sigma(\T^d))\cap \X^\sigma_T$ provided by Lemma~\ref{LEM:LWP} satisfies
\begin{align}\label{energy}
\|u\|_{\Eb^\sigma_T}^2\les \|u_0\|_{H^\sigma}^2+T\|u\|_{\Fb^{\sigma,\nu}_T}^2\|u\|_{\Fb^{s,\nu}_T}^{m-1}
\end{align}
for any $\sigma\ge s$.
\end{lemma}
\begin{proof}
Since $u_0\in H^\infty(\T^d)$, the mild solution $u$ with initial data $u_0$ provided by Lemma~\ref{LEM:LWP} satisfies $u\in C([0;T);H^\sigma(\T^d))\cap \X^\sigma_T$ for any $\sigma\ge s_\RR$, with $T=T(\|u_0\|_{H^{s_\RR}})\in(0;1]$. In particular,
\begin{align*}
v(t):=e^{W_t\L}u(t)
\end{align*}
belongs to $C([0;T);H^\infty(\T^d))$ and satisfies the integral equation
\begin{align*}
v(t)=e^{W_0\L}u_0+\int_0^te^{W_{\tau}\L}\NN\big(e^{-W_{\tau}\L}v\big)d\tau
\end{align*}
in $C([0;T);H^\sigma(\T^d))$ for any $\sigma\ge s_\RR$.

Since for any $k_0\in\Z^d$ it holds
\begin{align*}
\ft v(t,k_0)=e^{iW_0\varphi(k_0)}\ft u_0(k_0)+\int_0^t\widehat{\big(e^{W_\tau\L}\NN\big(e^{-W_\tau\L}v\big)\big)}(k_0)d\tau
\end{align*}
in $C([0;T);\C)$, and since we have the bound
\begin{align*}
\Big|\widehat{\big(e^{W_\tau\L}\NN\big(e^{-W_\tau\L}v\big)\big)}(k_0)\Big|&\les \sum_{\substack{(k_1,\dots,k_m)\in(\Z^d)^m\\\sum_{j=0}^m\pm_jk_j=0}}\jb{k_0}^{\beta_0}\prod_{j=1}^m\jb{k_j}^{\beta_j}|\ft v(\tau,k_j)| +\jb{k_0}^{|\beta_\RR|}|\ft v(\tau,k_0)|^{m}\\
&\les \jb{k_0}^{\beta_0} \|v\|_{L^\infty_T H^{\sigma}}^m+\|v\|_{L^\infty_T H^{\sigma}}^m,
\end{align*}
for $\sigma>\max(s_\RR;\max_j \beta_j+\frac{d}2)$, and with a constant uniform in $\tau$ and $k_0$, together with the continuity of $W_t$, then Lebesgue dominated convergence theorem ensures that $\tau\mapsto \widehat{\big(e^{W_\tau\L}\NN\big(e^{-W_\tau\L}v\big)\big)}(k_0)$ is continuous on $(0;T)$, and thus  $t\mapsto \ft v(t,k_0)$ is $C^1((0;T);\C)$ with
\begin{multline*}
\dt \ft v(t,k_0)= \sum_{\substack{(k_1,\dots,k_m)\in(\Z^d)^m\\\sum_{j=0}^m\pm_jk_j=0}}\ft \NN_0(k_0,\dots,k_m)e^{iW_t\big(\sum_{j=0}^m\pm_j\varphi(k_j)\big)}\prod_{j=1}^m\ft v_{k_j}(t)^{\pm_j}\\+\ft\RR(k_0,\dots,k_0)|\ft v_{k_0}(t)|^{m-1}\ft v_{k_0}(t),
\end{multline*}
due to \eqref{Resonant}.

Now, using also assumption \ref{A5*}$_{\alpha,\beta,\beta_\RR}$, we infer that
\begin{align*}
\frac{d}{dt}\frac12|\ft v_{k_0}(t)|^2
&=\Re\Big\{\sum_{\substack{(k_1,\dots,k_m)\in(\Z^d)^m\\\sum_{j=0}^m\pm_jk_j=0}}\ft \NN_0(k_0,\dots,k_m)e^{iW_t\sum_{j=0}^m\pm_j\varphi(k_j)}\cj{\ft v_{k_0}}(t)\prod_{j=1}^m\ft v_{k_j}(t)^{\pm_j}\\&\qquad\qquad+\ft\RR(k_0,\dots,k_0)|\ft v_{k_0}(t)|^{m+1}\Big\}\\
&=\Re\Big\{\sum_{\substack{(k_1,\dots,k_m)\in(\Z^d)^m\\\sum_{j=0}^m\pm_jk_j=0}}\ft \NN_0(k_0,\dots,k_m)e^{iW_t\sum_{j=0}^m\pm_j\varphi(k_j)}\prod_{j=0}^m\ft v_{k_j}(t)^{\pm_j}\Big\},
\end{align*} 
where the cancellation of the completely resonant nonlinear part comes from \eqref{Resonant}.

Integrating in time, we find
\begin{align}\label{actions}
&\frac12|\ft u_{k_0}(t)|^2\notag\\
&=\frac12|\ft v_{k_0}(t)|^2\notag\\
&=\frac12|(\ft u_0)_{k_0}|^2+\Re\Big\{\int_0^t\sum_{\substack{(k_1,\dots,k_m)\in(\Z^d)^m\\\sum_{j=0}^m\pm_jk_j=0}}\ft \NN_0(k_0,\dots,k_m)e^{iW_\tau\sum_{j=0}^m\pm_j\varphi(k_j)}\prod_{j=0}^m\ft v_{k_j}(\tau)^{\pm_j}d\tau\Big\}.
\end{align}
Summing in $k_0\in\Z^d$, this yields for any dyadic integer $N$:
\begin{align*}
&\frac12\|\P_{N_0}u\|_{L^\infty_TL^2_x}^2\le\frac12\|\P_{N_0}u_0\|_{L^2_x}^2+\sup_{t\in[0;T]}\Big|\int_0^t\int_{\T^d}\jb{\nabla}^\sigma\P_{N_0}\NN_0(u)\cj{\jb{\nabla}^\sigma \P_{N_0}u}dxd\tau\Big|\\
&\le\frac12\|\P_{N_0}u_0\|_{L^2_x}^2\\
&\qquad\qquad+\sup_{t\in[0;T]}\sum_{\substack{N_1,\dots,N_m\\ N_0^*\sim N_1^*}}N_0^{2\sigma}\sum_{\substack{I_{N_0^*}\subset [0;t)\\|I_{N_0^*}|=(N_0^*)^{-\nu}\wedge t\\\sqcup_\ell I_{N_0^*}=[0;t)}}\Big|\int_{I_{N_0^*}}\int_{\T^d}\NN_0(\P_{N_1 u},\dots,\P_{N_m}u)\cj{\P_{N_0}u}dxd\tau\Big|\\
&\le \frac12\|\P_{N_0}u_0\|_{L^2_x}^2\\
&\qquad+CT\sum_{\substack{N_1,\dots,N_m\\ N_0^*\sim N_1^*}}N_0^{2\sigma}\sup_{\substack{I_{N_0^*}\subset [0;t)\\|I_{N_0^*}|=(N_0^*)^{-\nu}\wedge t}}(N_0^*)^\nu\Big|\int_{I_{N_0^*}}\int_{\T^d}\NN_0(\P_{N_1 u},\dots,\P_{N_m}u)\cj{\P_{N_0}u}dxd\tau\Big|.
\end{align*}
Then, repeating the proof of the estimate on $\NN_0$ in Lemma~\ref{LEM:N}, we finally find
\begin{align*}
\|u\|_{\Eb^\sigma_T}^2\les \|u_0\|_{H^\sigma}^2 +T\|u\|_{\Fb^{\sigma,\nu}_T}^2\|u\|_{\Fb^{s,\nu}_T}^{m-1},
\end{align*}
provided now that
%\begin{align*}
%\rho\ge \alpha_1^{-1}\big(\wt\beta_0^*+\wt\beta_1^*\big)\qquad\text{and}\qquad \rho>(\alpha_1+\alpha_2)^{-1}\big(\wt\beta_0^*+\wt\beta_1^*+\sum_{j\ge 2}\max\big(\wt\beta_j^*+\frac{d}2-s;0\big)\big),
%\end{align*}
%and
%\begin{align*}
%\rho>\alpha_1^{-1}\big(\wt\beta_0^*+\wt\beta_1^*-2s\big)\qquad\text{and}\qquad \rho>(\alpha_1+\alpha_2)^{-1}\big(\wt\beta_0^*+\wt\beta_1^*+\sum_{j\ge 2}\max\big(\wt\beta_j^*+\frac{d}2-s;0\big)\big),
%\end{align*}
%where now $\wt\beta=B_{(1-\gamma)\nu}A_\nu (\beta)$.
\begin{align*}
\begin{cases}
\rho> \alpha_1^{-1}\big(\beta_0+\beta_1^*+(1-\gamma)\nu),\\
\rho> \alpha_1^{-1}\big(\beta_1^*+\beta_2^*-2s+(1-\gamma)\nu),\\
{\displaystyle \rho>(\alpha_1+\alpha_2)^{-1}\big(\beta_1^*+\beta_2^*-2s+\sum_{\ell=2}^{L}(\beta_{\ell+1}^*+\frac{d}2-s)\big)+(1-\gamma)\nu,}\\
\qquad\qquad\qquad\qquad \text{ for any } 2\le \ell_0\le m\text{ and }2\le L\le \ell_0-1,\\
{\displaystyle \rho>(\alpha_1+\alpha_2)^{-1}\big(\beta_1^*+\beta_2^*+\beta_0+\frac{d}2-s+\sum_{\ell=2}^{L}(\beta_{\ell+1}^*+\frac{d}2-s)\big)+(1-\gamma)\nu},\\
\qquad\qquad\qquad\qquad\text{ for any }2\le \ell_0\le m\text{ and }\ell_0+1\le L\le m,
\end{cases}
\end{align*}
since we do not gain anymore the factor $N_0^{-\nu}$ compared to the proof of Lemma~\ref{LEM:N}.

This proves \eqref{energy}.
\end{proof}
\begin{remark}\label{REM:smoothing}
\rm
From the proof of \eqref{energy} above, when $\sigma=s$, since the conditions \eqref{nu} on $\nu$ and \ref{Arho}$_{\alpha,\wt\beta,s}$ on $\rho$, with $\wt\beta=B_{(1-\gamma)\nu}A_\nu(\beta)$, are open, given $s,\nu$ and $\rho$ as in Lemma~\ref{LEM:energy}, there is $0<\delta(s,\alpha,\beta,\nu,\rho)\ll 1$ such that $\Fb^{s,\nu}_T$ in the right-hand side of \eqref{energy} can be replaced by $\Fb^{s-\delta,\nu}_T$. This smoothing property will be crucial in the proof of Theorem~\ref{THM:LWP-R} below, as in \cite[Remark 7.7]{GO}.
\end{remark}

We can now establish the existence part of Theorem~\ref{THM:LWP-R}~(ii).
\begin{proof}[Proof of Theorem~\ref{THM:LWP-R}~(ii)] 
Collecting \eqref{linear}-\eqref{N}-\eqref{energy} and letting $$X_\sigma(T):=\|u\|_{\Eb^\sigma_T}+\|\NN(u)\|_{\Nb^{\sigma,\nu}_T},$$ we find
\begin{align}\label{apriori1}
X_s(T)^2\le C\|u_0\|_{H^s}^2+CTX_s(T)^{m+1}+CT^{2\theta}X_s(T)^{2m}.
\end{align}
Since $T\mapsto X_s(T)$ is continuous and $$\lim_{T\searrow 0}\|u\|_{\Eb^s_T}\le C'\|u_0\|_{H^s}\qquad\text{and}\qquad\|\NN(u)\|_{\Nb^{s,\nu}_T}\underset{T\searrow 0}{\longrightarrow}0,$$ a continuity argument on \eqref{apriori1} provides 
\begin{align}\label{T0}
T_0=T_0(\|u_0\|_{H^s})
\end{align}
such that
\begin{align}\label{apriori2}
\|u\|_{\Fb^{s,\nu}_T}\les X_s(T)\les \|u_0\|_{H^s}
\end{align}
uniformly in $T\le T_0$, for any smooth solution emanating from $u_0\in H^\infty(\T^d)$.

Now, let $u_0\in H^s(\T^d)$ and $u_{0,n}=\P_{\le n}u_0\in H^\infty(\T^d)$, and let $u_n$ be the smooth solution emanating from $u_{0,n}$, which is also defined on $[0;T_0)$. From \eqref{apriori2} and \eqref{embedding}, we have that $\{u_n\}$ is uniformly bounded in $C([0;T];H^s(\T^d))$ for any $T<T_0$. Moreover, by Remark~\ref{REM:smoothing}, we find that \eqref{energy} holds with $\Fb^{s,\nu}_T$ replaced by $\Fb^{s-\delta,\nu}_T$ in the right-hand side, for some $0<\delta\ll 1$. Then deriving the energy estimate for $\P_{>N}u$ instead of $u$ in Lemma~\ref{LEM:energy} yields
\begin{align*}
\|\P_{> N}u_n\|_{L^\infty_TH^s}^2&\les \|\P_{> N}u_n\|_{\Eb^s_T}^2\\
&\les \|\P_{> N}u_{0,n}\|_{H^s}^2+T\|\P_{>N}u_n\|_{\Fb^{s-\delta,\nu}_T}^2\|u_n\|_{\Fb^{s-\delta,\nu}_T}^{m-1}.
\end{align*}
Since \eqref{apriori2} is uniform in $n$ and $N$, using our choice of $T_0(\|u_0\|_{H^s})$ with $T<T_0$, we get
\begin{align}\label{tail}
\|\P_{> N}u_n\|_{L^\infty_TH^s}^2\les\|\P_{> N}u_n\|_{\Eb^s_T}^2\les \|\P_{> N}u_{0,n}\|_{H^s}^2+C(\|u_0\|_{H^s})N^{-\delta}\underset{N\to\infty}{\longrightarrow}0,
\end{align}
uniformly in $n$ (recall that $u_{0,n}=\P_{\le n}u_0$).

Taking arbitrary $\eps>0$, we can thus find $N_0$ such that $\|\P_{> N_0}u_n\|_{L^\infty_TH^s}\le \frac{\eps}{3}$. Then, for any $0<h<1$, writing the Duhamel formula on $[t;t+h]$, we can estimate
\begin{align*}
&\|\P_{\le N_0}u_n(t+h)-\P_{\le N_0}u_n(t)\|_{H^s}\\
&\le\Big\|\big(e^{-(W_{t+h}-W_t)\L}-1\big)\P_{\le N_0}u_n(t)\Big\|_{H^s}+\Big\|\int_t^{t+h}e^{-(W_{t+h}-W_\tau)\L}\P_{\le N_0}\NN(u_n)(\tau)d\tau\Big\|_{H^s}\\
&\le \omega^W_T(h)\|\varphi\|_{L^\infty(|k|\le N_0)}\big\|\P_{\le N_0}u_n\big\|_{L^\infty_T H^s}+\sum_{\wt N_0\le N_0}\sum_{\substack{I_{\wt N_0,\ell}\subset[t;t+h)\\|I_{\wt N_0,\ell}|=\wt N_0^{-\nu}\wedge h\\\cup_\ell I_{\wt N_0,\ell}=[t;t+h)}}\big\|\P_{\wt N_0}\NN(u_n)\big\|_{DU^2_{W\L}(I_{\wt N_0,\ell})H^s},
\end{align*}
where $\omega^W_T$ is the uniform modulus of continuity of $W_t$ on $[0;T]$. Moreover, repeating the proof of Lemma~\ref{LEM:energy}, we have that
\begin{align*}
&\sum_{\wt N_0\le N_0}\sum_{\substack{I_{\wt N_0,\ell}\subset[t;t+h)\\|I_{\wt N_0,\ell}|=\wt N_0^{-\nu}\wedge h\\\sqcup_\ell I_{\wt N_0,\ell}=[t;t+h)}}\big\|\P_{\wt N_0}\NN_0(u_n)\big\|_{DU^2_{W\L}(I_{\wt N_0,\ell})H^s}\les h \|\P_{\le N_0}u_n\|_{\Fb^{s,\nu}_T}\|u_n\|_{\Fb^{s,\nu}_T}^{m-1},
\end{align*}
while repeating the proof of Lemma~\ref{LEM:N} with $\theta=0$ gives
\begin{align*}
&\sum_{\wt N_0\le N_0}\sum_{\substack{I_{\wt N_0,\ell}\subset[t;t+h)\\|I_{\wt N_0,\ell}|=\wt N_0^{-\nu}\wedge h\\\sqcup_\ell I_{\wt N_0,\ell}=[t;t+h)}}\big\|\P_{\wt N_0}\RR(u_n)\big\|_{DU^2_{W\L}(I_{\wt N_0,\ell})H^s}\\
&\qquad\qquad\les h N_0^\nu \sum_{\wt N_0\le N_0}\sup_{\substack{I_{\wt N_0,\ell}\subset[t;t+h)\\|I_{\wt N_0,\ell}|=\wt N_0^{-\nu}\wedge h\\\sqcup_\ell I_{\wt N_0,\ell}=[t;t+h)}}\big\|\P_{\wt N_0}\RR(u_n)\big\|_{DU^2_{W\L}(I_{\wt N_0,\ell})H^s}\\
&\qquad\qquad\les h N_0^\nu \|\P_{\le N_0}u\|_{\Fb^{s,\nu}_T}^m.
\end{align*}

Combining these estimates together with \eqref{apriori2} leads to
\begin{align*}
\|\P_{\le N_0}u_n(t+h)-\P_{\le N_0}u_n(t)\|_{H^s}\le C(N_0,\|u_0\|_{H^s})\big(\omega^W_T(h)+h).
\end{align*}
This shows that $\{\P_{\le N_0}u_n(t)\}_n$ is equicontinuous with values in $H^s(\T^d)$. Thus Arzelà–Ascoli theorem ensures that $\{\P_{\le N_0}u_n\}_n$ is pre-compact in $C([0;T];H^s(\T^d))$ for any $T<T_0$, which in particular implies that we can find a finite covering ${\displaystyle\big\{B_{C_TH^s}\big(\P_{\le N_0}u_{n_\ell};\frac{\eps}{3}\big)\big\}_{\ell=1,\dots,L}}$ of $\{\P_{\le N_0}u_n\}_n$ by balls in $C([0;T];H^s(\T^d))$ of radius $\frac{\eps}{3}$ centred at some $\P_{\le N_0}u_{n_\ell}$. From our choice of $N_0$, we finally deduce that ${\displaystyle\big\{B_{C_TH^s}\big(\P_{\le N_0}u_{n_\ell};\eps\big)\big\}_{\ell=1,\dots,L}}$ is a finite covering of $\{u_n\}_n$ in $C([0;T];H^s(\T^d))$. This shows that $\{u_n\}_n$ is pre-compact in $C([0;T];H^s(\T^d))$ for any $T<T_0(\|u_0\|_{H^s})$.

Therefore, there is a subsequence, still denoted $\{u_n\}$, which converges in $C([0;T];H^s(\T^d))$ to some $u\in C([0;T];H^s(\T^d))$. The bound \eqref{tail} allows to upgrade this convergence to $\Eb^s_T$. From \eqref{linear}-\eqref{N} and the multilinearity of $\NN$ we have
\begin{align*}
\|u_n-u_{n'}\|_{\Fb^{s,\nu}_T}&\les \|u_n-u_{n'}\|_{\Eb^s_T}+\big\|\NN(u_n)-\NN(u_{n'})\big\|_{\Nb^{s,\nu}_T}\\
&\les \|u_n-u_{n'}\|_{\Eb^s_T}+(T^\gamma+T^\theta)\|u_n-u_{n'}\|_{\Fb^{s,\nu}_T}\big(\|u_n\|_{\Fb^{s,\nu}_T}+\|u_{n'}\|_{\Fb^{s,\nu}_T}\big)^{m-1},
\end{align*}
which, together with \eqref{apriori2} and our choice of $T_0$ \eqref{T0}, yields
\begin{align*}
\|u_n-u_{n'}\|_{\Fb^{s,\nu}_T}&\les \|u_n-u_{n'}\|_{\Eb^s_T}.
\end{align*}
This shows that the subsequence $\{u_n\}$ also converges in $\Fb^{s,\nu}_T$, and together with Lemma~\ref{LEM:N} and the multilinearity of $\NN$ we also find
\begin{align*}
\big\|\NN(u_n)-\NN(u_{n'})\big\|_{\Nb^{s,\nu}_T} &\les (T^\gamma+T^\theta)\|u_n-u_{n'}\|_{\Fb^{s,\nu}_T}\big(\|u_n\|_{\Fb^{s,\nu}_T}+\|u_{n'}\|_{\Fb^{s,\nu}_T}\big)^{m-1}\\
& \les C(\|u_0\|_{H^s})\|u_n-u_{n'}\|_{\Fb^{s,\nu}_T},
\end{align*}
showing that $\{\NN(u_n)\}$ converges in $\Nb^{s,\nu}_T$. This ensures that the limit $u$ of the subsequence $\{u_n\}$ satisfies \eqref{EQ} and belongs to $C([0;T);H^s(\T^d))\cap \Fb^{s,\nu}_T\cap \Eb^s_T$ for any $T<T_0$. This concludes the proof of the existence part of Theorem~\ref{THM:LWP-R}~(ii).
\end{proof}

\subsection{Uniqueness and continuous dependence}
In this subsection, we finish the proof of Theorem~\ref{THM:LWP-R}. 

In order to prove uniqueness of the solution constructed in the previous subsection, we first establish an energy estimate for the difference of two solutions with the same initial datum.
\begin{lemma}\label{LEM:Ediff0} 
Let $s\in\R$, $u_0\in H^s(\T^d)$, and let us assume that for some $T_0>0$, there are two solutions $u_1,u_2\in C([0;T_0);H^s(\T^d))\cap\Fb^{s,\nu}_{T_0}\cap \Eb^s_{T_0}$ with same initial data $u_0$, and that these solutions are limits in $C([0;T_0);H^s(\T^d))\cap\Fb^{s,\nu}_{T_0}\cap \Eb^s_{T_0}$ of some smooth approximating solutions. Then, under assumptions \ref{A0}-\ref{A1}-\ref{A5*}$_{\alpha,\beta,\beta_\RR}$, provided that $\rho$ satisfies \ref{Arho*}$_{\alpha,\beta,\beta_\RR,s}$, it holds for any $T< T_0$:
\begin{align}\label{Ediff0}
&\|u_1-u_2\|_{\Eb^s_T}^2\notag\\
&\qquad\les T\|u_1-u_2\|_{\Fb^{s,\nu}_T}^2\big(\|u_1\|_{\Fb^{s,\nu}_T}+\|u_2\|_{\Fb^{s,\nu}_T}\big)^m\big(1+T\big(\|u_1\|_{\Fb^{s,\nu}_T}+\|u_2\|_{\Fb^{s,\nu}_T}\big)^{m-3}\big).
\end{align}
\end{lemma}
\begin{proof}
If $u_{j,n}$ is a smooth (in space) solution to \eqref{Duhamel} approximating $u_j$ in $C([0;T_0);H^s(\T^d))\cap\Fb^{s,\nu}_{T_0}\cap \Eb^s_{T_0}$, $j=1,2$, then $v_n:= u_{1,n}-u_{2,n}$ approximates $v=u_1-u_2$ in $C([0;T_0);H^s(\T^d))\cap\Fb^{s,\nu}_{T_0}\cap \Eb^s_{T_0}$. Moreover, repeating the proof of \eqref{actions} in Lemma~\ref{LEM:energy},
we find for any dyadic $N_0$ that
\begin{align}\label{energydyadic}
&\|\P_{N_0}v_n(t)\|_{L^2}^2\notag\\
&=\|\P_{N_0}v_n(0)\|_{L^2}^2+2\Re\Big\{\int_0^t \int_{\T^d}\P_{N_0}\big(\NN_0(u_{1,n})-\NN_0(u_{2,n})\big)(\tau)\cj{\P_{N_0}v_n}(\tau)dxd\tau\notag\\
&\qquad+\int_0^t \sum_{|k_0|\sim N_0}\ft\RR(k_0)\big(|(\ft u_{1,n})_{k_0}|^{m-1}(\ft u_{1,n})_{k_0}-|(\ft u_{2,n})_{k_0}|^{m-1}(\ft u_{2,n})_{k_0}\big)(\tau)\cj{\ft (v_n)}_{k_0}(\tau)d\tau\Big\}\notag\\
&=\|\P_{N_0}v_n(0)\|_{L^2}^2+2\Re\Big\{\int_0^t \int_{\T^d}\P_{N_0}\big(\NN_0(u_{1,n})-\NN_0(u_{2,n})\big)(\tau)\cj{\P_{N_0}v_n}(\tau)dxd\tau\notag\\
&\qquad+\int_0^t \sum_{|k_0|\sim N_0}\ft\RR(k_0)\big(|(\ft u_{1,n})_{k_0}|^{2}-|(\ft u_{2,n})_{k_0}|^{2}\big)(\tau)\notag\\
&\qquad\qquad\times\sum_{m'=0}^{\frac{m-3}2}|(\ft u_{1,n})_{k_0}(\tau)|^{2m'}|(\ft u_{2,n})_{k_0}(\tau)|^{m-3-2m'}(\ft u_{2,n})_{k_0}(\tau)\cj{\ft (v_n)_{k_0}}(\tau)d\tau\Big\}
\end{align}
where we used the last condition on $\RR$ in \ref{A5*}.

Similarly as in Lemma~\ref{LEM:energy}, we also have
\begin{align}\label{actions2}
|(\ft u_{j,n})_{k_0}(\tau)|^{2}=|(\ft u_{j,n})_{k_0}(0)|^{2}+2\Re\Big\{\int_0^\tau\widehat{\NN_0(u_{j,n})}_{k_0}(\tau')\cj{(\ft u_{j,n})_{k_0}}(\tau')d\tau'\Big\},~~j=1,2.
\end{align}

From the convergence of $u_{j,n}$ to $u_j$ in $C([0;T_0);H^s(\T^d))\cap \Eb^s_{T_0}\cap \Fb^{s,\nu}_{T_0}$, we can pass to the limit in \eqref{energydyadic} and \eqref{actions2} using estimates similar to the ones in Lemma~\ref{LEM:N}, providing the energy identities
\begin{align}\label{identity1}
&\|\P_{N_0}v(t)\|_{L^2}^2\notag\\
&=\|\P_{N_0}v(0)\|_{L^2}^2+2\Re\Big\{\int_0^t \int_{\T^d}\P_{N_0}\big(\NN_0(u_{1})-\NN_0(u_{2})\big)(\tau)\cj{\P_{N_0}v}(\tau)dxd\tau\notag\\
&\qquad+\int_0^t \sum_{|k_0|\sim N_0}\ft\RR(k_0)\big(|(\ft u_{1})_{k_0}|^{2}-|(\ft u_{2})_{k_0}|^{2}\big)(\tau)\notag\\
&\qquad\qquad\times\sum_{m'=0}^{\frac{m-3}2}|(\ft u_{1})_{k_0}(\tau)|^{2m'}|(\ft u_{2})_{k_0}(\tau)|^{m-3-2m'}(\ft u_{2})_{k_0}(\tau)\cj{\ft (v)_{k_0}}(\tau)d\tau\Big\}
\end{align}
and
\begin{align}\label{identity2}
|(\ft u_{j})_{k_0}(\tau)|^{2}=|(\ft u_{j})_{k_0}(0)|^{2}+2\Re\Big\{\int_0^\tau\widehat{\NN_0(u_{j})}_{k_0}(\tau')\cj{(\ft u_{j})_{k_0}}(\tau')d\tau'\Big\},~~j=1,2.
\end{align}

Since it also holds $|(\ft u_1)_{k_0}(0)|^{2}=|(\ft u_2)_{k_0}(0)|^{2}$, we can inject \eqref{identity2} into \eqref{identity1} to conclude that
\begin{align*}
&\|v(t)\|_{\Eb^s_T}^2\\
&\le 2\sum_{N_0}N_0^{2s}\sup_{t\in[0;T]}\Big|\int_0^t \int_{\T^d}\P_{N_0}\big(\NN_0(u_1)-\NN_0(u_2)\big)(\tau)\cj{\P_{N_0}v}(\tau)dxd\tau\Big|\\
&\qquad+\sum_{N_0}N_0^{2s}\sup_{t\in[0;T]}\Big|\int_0^t \sum_{|k_0|\sim N_0}\ft\RR(k_0)\\
&\qquad\qquad\qquad\times\Re\Big(\int_0^\tau\big(\widehat{\NN_0(u_1)}_{k_0}(\tau')\cj{(\ft u_1)_{k_0}}(\tau')-\widehat{\NN_0(u_2)}_{k_0}(\tau')\cj{(\ft u_2)_{k_0}}(\tau')\big)d\tau'\Big)\\
&\qquad\qquad\qquad\times\sum_{m'=0}^{\frac{m-3}2}|(\ft u_1)_{k_0}(\tau)|^{2m'}|(\ft u_2)_{k_0}(\tau)|^{m-3-2m'}(\ft u_2)_{k_0}(\tau)\cj{\ft v_{k_0}}(\tau)d\tau\Big|.
\end{align*}
Proceeding as in Lemma~\ref{LEM:energy}, we can estimate the contribution of the first time integral by
\begin{multline}\label{Ediff01}
\sum_{N_0}N_0^{2s}\sup_{t\in[0;T]}\Big|\int_0^t \int_{\T^d}\P_{N_0}\big(\NN_0(u_1)-\NN_0(u_2)\big)(\tau)\cj{\P_{N_0}v}(\tau)dxd\tau\Big|^2\\\les T\|v\|_{\Fb^{s,\nu}_T}^2\big(\|u_1\|_{\Fb^{s,\nu}_T}+\|u_2\|_{\Fb^{s,\nu}_T}\big)^{m}
\end{multline}
provided again that
\begin{align*}
\begin{cases}
\rho> \alpha_1^{-1}\big(\beta_0+\beta_1^*+(1-\gamma)\nu),\\
\rho> \alpha_1^{-1}\big(\beta_1^*+\beta_2^*-2s+(1-\gamma)\nu),\\
{\displaystyle \rho>(\alpha_1+\alpha_2)^{-1}\big(\beta_1^*+\beta_2^*-2s+\sum_{\ell=2}^{L}(\beta_{\ell+1}^*+\frac{d}2-s)\big)+(1-\gamma)\nu,}\\
\qquad\qquad\qquad\qquad \text{ for any } 2\le \ell_0\le m\text{ and }2\le L\le \ell_0-1,\\
{\displaystyle \rho>(\alpha_1+\alpha_2)^{-1}\big(\beta_1^*+\beta_2^*+\beta_0+\frac{d}2-s+\sum_{\ell=2}^{L}(\beta_{\ell+1}^*+\frac{d}2-s)\big)+(1-\gamma)\nu},\\
\qquad\qquad\qquad\qquad\text{ for any }2\le \ell_0\le m\text{ and }\ell_0+1\le L\le m.
\end{cases}
\end{align*}
%\begin{align*}
%\rho\ge \alpha_1^{-1}\big(\wt\beta_0^*+\wt\beta_1^*\big)\qquad\text{and}\qquad \rho>(\alpha_1+\alpha_2)^{-1}\big(\wt\beta_0^*+\wt\beta_1^*+\sum_{j\ge 2}\max\big(\wt\beta_j^*+\frac{d}2-s;0\big)\big),
%\end{align*}
%and
%\begin{align*}
%\rho>\alpha_1^{-1}\big(\wt\beta_0^*+\wt\beta_1^*-2s\big)\qquad\text{and}\qquad \rho>(\alpha_1+\alpha_2)^{-1}\big(\wt\beta_0^*+\wt\beta_1^*+\sum_{j\ge 2}\max\big(\wt\beta_j^*+\frac{d}2-s;0\big)\big),
%\end{align*}
%with $\wt\beta=B_{(1-\gamma)\nu}A_\nu (\beta)$.

As for the contribution of the second time integral, we can proceed again as in the proof of Lemma~\ref{LEM:energy}, saving $(N_0\wedge N_2^*)^{-\frac{d}2}$ derivatives in the $\tau'$ integral thanks to the resonant structure\footnote{we estimate the $\ell^\infty_{k_0}$ norm instead of the $\ell^2_{k_0}$ one.}:
\begin{align*}
&\sum_{N_0}N_0^{2s}\sup_{t\in[0;T]}\Big|\int_0^t \sum_{|k_0|\sim N_0}\ft\RR(k_0)2\Re\Big(\int_0^\tau\big(\widehat{\NN_0(u_1)}_{k_0}(\tau')\cj{(\ft u_1)_{k_0}}(\tau')-\widehat{\NN_0(u_2)}_{k_0}(\tau')\cj{(\ft u_2)_{k_0}}(\tau')\big)d\tau'\Big)\Big|\\
&\qquad\qquad\times\sum_{m'=0}^{\frac{m-3}2}|(\ft u_1)_{k_0}(\tau)|^{2m'}|(\ft u_2)_{k_0}(\tau)|^{m-3-2m'}(\ft u_2)_{k_0}(\tau)\cj{\ft v_{k_0}}(\tau)d\tau\\
&\les TN_0^{2s+|\beta_\RR|+\nu}\sup_{\substack{I_{N_0}\subset[0;T)\\|I_{N_0}|=N_0^{-\nu}\wedge T}}\|\P_{N_0}v\|_{U^2_{W\L}(I_{N_0})L^2}\big(\|\P_{N_0}u_1\|_{U^2_{W\L}(I_{N_0})L^2}+\|\P_{N_0}u_2\|_{U^2_{W\L}(I_{N_0})L^2}\big)^{m-3}\\
&\qquad\qquad\times\sup_{\tau\in I_{N_0}}\sup_{|k_0|\sim N_0}\Big|\int_0^\tau\big(\widehat{\NN_0(u_1)}_{k_0}(\tau')\cj{(\ft u_1)_{k_0}}(\tau')-\widehat{\NN_0(u_2)}_{k_0}(\tau')\cj{(\ft u_2)_{k_0}}(\tau')\big)d\tau'\Big|\\
&\les T\sum_{N_0}N_0^{2s+|\beta_\RR|+\nu}\sup_{\substack{I_{N_0}\subset[0;T)\\|I_{N_0}|=N_0^{-\nu}\wedge T}}\|\P_{N_0}v\|_{U^2_{W\L}(I_{N_0})L^2}\big(\|\P_{N_0}u_1\|_{U^2_{W\L}(I_{N_0})L^2}+\|\P_{N_0}u_2\|_{U^2_{W\L}(I_{N_0})L^2}\big)^{m-3}\\
&\qquad\qquad\times\sup_{\tau\in I_{N_0}}\sum_{\substack{N_1,\dots,N_m\\N_0^*\sim N_1^*}}\sup_{\substack{J_{N_0^*}\subset[0;\tau)\\|J_{N_0^*}|=(N_0^*)^{-\nu}\wedge \tau}}\tau (N_0^*)^{(1-\gamma)\nu} (N_0\wedge N_2^*)^{-\frac{d}2}(N_0^*)^{\beta_{j_0}+\beta_{j_1}-\alpha_1\rho}(N_2^*)^{\beta_{j_2}-\alpha_2\rho+\frac{d}2}\\
&\qquad\times\Big(\prod_{\ell=3}^m(N_{j_\ell})^{\beta_{j_\ell}+\frac{d}2}\Big)\sum_{j=0}^m\|\P_{N_j}v\|_{U^2_{W\L}(J_{N_0^*})L^2}\prod_{\substack{j'=0\\j'\neq j}}^m\big(\|\P_{N_{j'}}u_1\|_{U^2_{W\L}(J_{N_0^*})L^2}+\|\P_{N_{j'}}u_2\|_{U^2_{W\L}(J_{N_0^*})L^2}\big).
\end{align*}
After summing on the $N_j$ as in the proof of Lemma~\ref{LEM:N}, we finally find the estimate for the second time integral as
\begin{align}\label{Ediff02}
&\sum_{N_0}N_0^{2s}\sup_{t\in[0;T]}\Big|\int_0^t \sum_{|k_0|\sim N_0}\ft\RR(k_0)2\Re\Big(\int_0^\tau\big(\widehat{\NN_0(u_1)}_{k_0}(\tau')\cj{(\ft u_1)_{k_0}}(\tau')-\widehat{\NN_0(u_2)}_{k_0}(\tau')\cj{(\ft u_2)_{k_0}}(\tau')\big)d\tau'\Big)\notag\\
&\qquad\qquad\times\sum_{m'=0}^{\frac{m-3}2}|(\ft u_1)_{k_0}(\tau)|^{2m'}|(\ft u_2)_{k_0}(\tau)|^{m-3-2m'}(\ft u_2)_{k_0}(\tau)\cj{\ft v_{k_0}}(\tau)d\tau\Big|\notag\\
&\qquad\qquad\qquad\qquad\les T^2\|v\|_{\Fb^{s,\nu}_T}^2\big(\|u_1\|_{\Fb^{s,\nu}_T}+\|u_2\|_{\Fb^{s,\nu}_T}\big)^{2m-3},
\end{align}
provided now that
\begin{align*}
\begin{cases}
\rho\ge  \alpha_1^{-1}\big(\beta_0+\beta_1^*+|\beta_\RR|+(2-\gamma)\nu\big),\\
\rho\ge  \alpha_1^{-1}\big(\beta_1^*+\beta_2^*-2s+(1-\gamma)\nu),\\
{\displaystyle \rho>(\alpha_1+\alpha_2)^{-1}\big(\beta_1^*+\beta_2^*-2s+\sum_{\ell=2}^{L}(\beta_{\ell+1}^*+\frac{d}2-s)+(1-\gamma)\nu\big),}\\
\qquad\qquad\qquad\qquad \text{ for any } 2\le \ell_0\le m\text{ and }2\le L\le \ell_0-1,\\
{\displaystyle \rho>(\alpha_1+\alpha_2)^{-1}\big(\beta_1^*+\beta_2^*+\beta_0+|\beta_\RR|-s+\sum_{\ell=2}^{L}(\beta_{\ell+1}^*+\frac{d}2-s)+(2-\gamma)\nu\big)},\\
\qquad\qquad\qquad\qquad\text{ for any }2\le \ell_0\le m\text{ and }\ell_0+1\le L\le m,
\end{cases}
\end{align*}
Combining \eqref{Ediff01} and \eqref{Ediff02} finally proves \eqref{Ediff0}.
\end{proof}
\begin{proof}[Proof of Theorem~\ref{THM:LWP-R}~(iii)]
With the energy estimate of Lemma~\ref{LEM:Ediff0}, we can now show the uniqueness of solutions claimed in Theorem~\ref{THM:LWP-R}~(iii). Indeed, if there are two solutions $u_1,u_2\in C([0;T_0);H^s(\T^d))\cap \Eb^s_{T_0}\cap \Fb^{s,\nu}_{T_0}$ sharing the same initial datum $u_0\in H^s(\T^d)$, letting $$R=\max(\|u_1\|_{\Fb^{s,\nu}_{T_0}};\|u_2\|_{\Fb^{s,\nu}_{T_0}}),$$
 applying \eqref{linear} and Lemmas~\ref{LEM:N} and~\ref{LEM:Ediff0}, we find that for $T=\min(A^{-1}\jb{R}^{-\eta};T_0)$ for $\eta=\max(\frac{m}2;m-3;\frac{m-1}{\theta})$ and $A\gg 1$, similarly as in \eqref{apriori1} it holds
\begin{align*}
\|u_1-u_2\|_{\Fb^{s,\nu}_T} &\les \|u_1-u_2\|_{\Eb^s_T} + \|\NN(u_1)-\NN(u_2)\|_{\Nb^{s,\nu}_T}\\
&\le CT\|u_1-u_2\|_{\Fb^{s,\nu}_T}R^{\frac{m}2}\big(1+TR^{m-3}\big)^\frac12 + CT^\theta R^{m-1}\|u_1-u_2\|_{\Fb^{s,\nu}_T}\\
& \le \frac12\|u_1-u_2\|_{\Fb^{s,\nu}_T}.
\end{align*}
This shows that $u_1=u_2$ on $[0;T]$. Since our choice of $T$ only depends on $\|u_1\|_{\Fb^{s,\nu}_{T_0}}$ and $\|u_2\|_{\Fb^{s,\nu}_{T_0}}$, we can iterate this argument on $[T;2T]$, than $[2T;3T]$ etc. as long as $kT\le T_0$. This finally yields $u_1=u_2$ on $[0;T_0]$.
\end{proof}

\begin{proof}[~End of the proof of Theorem~\ref{THM:LWP-R}]
With the existence and uniqueness statements of Theorem~\ref{THM:LWP-R}~(ii) and~(iii), we can now prove the continuous dependence in the initial datum by repeating the argument in the previous subsection leading to the existence part of Theorem~\ref{THM:LWP-R}.

Indeed, if $u_0\in H^s(\T^d)$, let $u_{0,n}\in H^s(\T^d)$ be such that $$u_{0,n}\to u_0\text{ in }H^s(\T^d).$$
Let then $u\in C([0;T_0);H^s(\T^d))\cap \Eb^s_{T_0}\cap \Fb^{s,\nu}_{T_0}$ be the unique solution associated with $u_0$ constructed previously, with $T_0$ as in \eqref{T0}. Without loss of generality, we can again assume that $\|u_{0,n}\|_{H^s}\le \|u_0\|_{H^s}+1$, and up to taking $T_0$ slightly smaller in \eqref{T0} if necessary, we have that the unique solution $u_n\in C([0;T_0);H^s(\T^d))\cap \Eb^s_{T_0}\cap \Fb^{s,\nu}_{T_0}$ starting from $u_{0,n}$ is also defined on $[0;T_0)$. Then, from the compactness of $\{u_{0,n}\}_n\cup \{u_0\}$ in $H^s(\T^d)$, for all $\eps>0$, there is $N_0\in\N$ such that 
$$\|\P_{>N_0}u_0\|_{H^s}+\sup_n\|\P_{>N_0}u_{0,n}\|_{H^s}<\eps.$$
Then proceeding as in the proof of the existence, we find again from the energy estimate of Lemma~\ref{LEM:energy} together with Remark~\ref{REM:smoothing} that for any $T<T_0$,
\begin{align}\label{tail2}
\|\P_{>N_0}u_n\|_{C_TH^s}^2\les \|\P_{>N_0}u_n\|_{\Eb^s_T}\les \|\P_{N_0}u_{0,n}\|_{H^s}+C(\|u_0\|_{H^s})N^{-2\delta} \les \eps,
\end{align}
eventually taking $N_0$ larger if necessary. Repeating the argument in the proof of existence, we deduce that $\{u_n\}_n$ is precompact in $C([0;T];H^s(\T^d))$, thus having a subsequence converging in $C([0;T];H^s(\T^d))$. Again, from the tail estimate \eqref{tail2}, this convergence is upgraded in $\Eb^s_T$, and using again Lemmas~\ref{LEM:FNE} and~\ref{LEM:N} we deduce that it also converges in $\Fb^s_T$. From the uniqueness statement already established, we get that the limit of the subsequence must be $u$, and that the whole sequence converges. This concludes the proof of Theorem~\ref{THM:LWP-R}~(iv).

%As for the globalization in Theorem~\ref{THM:LWP-R}~(v), with the assumptions \ref{A2} and \ref{Arho}$_{\alpha,\beta,0}$, local well-posedness in $L^2(\T^d)$ is established by the previous arguments, with a local time of existence only depending on $\|u_0\|_{L^2}$. Since the $L^2$ norm is a conserved quantity under \ref{A2}, we can iterate the local well-posedness indefinitely to obtain existence and uniqueness on arbitrary time intervals.
\end{proof}

%\section{Almost sure globalization and invariance of the white noise}
%
%+Proof of Example~\ref{EX:global}
%
%\appendix

\end{document}